\documentclass[12pt]{amsart}

\usepackage[hmargin=0.8in,twosideshift=0in,height=8.6in]{geometry}
\usepackage{amssymb,amsthm}
\usepackage{delarray}
\usepackage{natbib}
\renewcommand{\cite}{\citet}

\usepackage{ifpdf}
\ifpdf
\usepackage[pdftex]{graphicx}
\else
\usepackage[dvips]{graphicx}
\fi

\usepackage{ifpdf}
\usepackage{color}
\definecolor{webgreen}{rgb}{0,.5,0}
\definecolor{webbrown}{rgb}{.8,0,0}
\definecolor{emphcolor}{rgb}{0.95,0.95,0.95}

\usepackage{hyperref}
\hypersetup{%
          colorlinks=true,
          linkcolor=webbrown,
          filecolor=webbrown,
          citecolor=webgreen,
          breaklinks=true}
\ifpdf
\hypersetup{pdftex,
            pdfstartview=FitH, %%Fit, FitB, FitH
            bookmarksopen=true,
            bookmarksnumbered=true
}
\else
\hypersetup{dvips}
\fi

\renewcommand{\theequation}{\thesection.\arabic{equation}}
\numberwithin{equation}{section} 

\newtheorem {prop}{Proposition}[section]
\newtheorem {lemm}{Lemma}[section]

\newtheorem {cor}{Corollary}[section]

\theoremstyle{remark}
\newtheorem {rem}{Remark}[section]

\newcommand{\R}{\mathbb R}
\newcommand{\N}{\mathbb N}
\newcommand{\PP}{\mathbb P}
\renewcommand{\P}{\mathbb P}
\newcommand{\A}{\mathcal A}
\newcommand{\E}{\mathbb E}

\newcommand{\M}{\mathcal M}
\newcommand{\tM}{\tilde{\mathcal M}}
\newcommand{\tU}{\tilde{U}_0}
\newcommand{\U}{\mathcal U}

\newcommand{\F}{\mathbb F}

\renewcommand{\bar}{\overline}

\newcommand{\s}{\mathcal{S}}
\renewcommand{\emptyset}{\varnothing}
\newcommand{\Fc}{\mathcal F}
\newcommand{\vP}{\vec{\Pi} }
\newcommand{\vp}{\vec{\pi} }
\newcommand{\vx}{\vec{x} }
\newcommand{\eps}{\varepsilon}
\newcommand{\e}{\mathrm{e}}
\newcommand{\G}{\mathcal{G}}
\renewcommand{\S}{\mathcal{S}}

\title[Sequential Tracking of a Hidden Markov Chain ]{Sequential Tracking of a Hidden Markov Chain Using Point Process Observations}

\author{Erhan Bayraktar }
\address[E. Bayraktar]{Department of
  Mathematics, University of Michigan, Ann Arbor, MI 48109}
\email{erhan@umich.edu}
\thanks{E. Bayraktar is supported in part by the National Science Foundation, under grant DMS-0604491. }

\author{Michael Ludkovski}
\address[M.\ Ludkovski]{Department of Mathematics
  University of Michigan, Ann Arbor, MI 48109}
\email{mludkov@umich.edu}


\begin{document}
\begin{abstract}
We study finite horizon optimal switching problems for hidden Markov chain models with point
process observations. The controller possesses a finite range of strategies and attempts to
track the state of the unobserved state variable using Bayesian updates over the discrete
observations. Such a model has applications in economic policy making, staffing under variable
demand levels and generalized Poisson disorder problems. We show regularity of the value
function and
 explicitly characterize an optimal strategy.  We also provide an efficient numerical scheme and illustrate our results with
several computational examples.
\end{abstract}

\maketitle
\section{Introduction}\label{sec:intro}
An economic agent (henceforth the controller) observes a compound Poisson process $X$ with
arrival rate $\lambda$, and mark/jump distribution $\nu$. The local characteristics $(\lambda,
\nu)$ of $X$ are determined by the current state of an \emph{unobservable} Markov jump process
$M$ with finite state space $E \triangleq \{ 1, \ldots, m\}$. More precisely, the
characteristics are $(\lambda_i , \nu_i)$ whenever $M$ is at state $i$, for $i \in E $.

The objective of the controller is to \emph{track} the state of $M$ given the information in
$X$. To do so, the controller possesses a range of policies $a$ in the finite alphabet $
\mathcal{A} \triangleq \{ 1, \ldots, A \}$. The policies are sequentially adopted starting from
time 0 and until some fixed horizon $T< \infty$. The infinite horizon case $T=+\infty$ is
treated in Section \ref{sec:infinite-horizon}. The selected policy $a$ leads to running costs
(benefits) at instantaneous rate
\begin{align*}
\sum_{i \in E }  c_{i}(a) 1_{ \{ M_{t} = i  \} } dt.
\end{align*}

%Because $K$'s are strictly positive, policy changes are expensive and the
%controller will only carry out a finite number of switches before $T$ (since potential rewards
%are bounded).
%Also, we will assume that
%$K_i(a,b)+K_i(b,c) \geq K_i(a,c)$, for any $a,b,c \in \A$, i.e. direct transitions are less costly.

The controller's overall strategy consists of a double sequence $(\tau_k, \xi_k), k
=0,1,2,\ldots$, with $\xi_k\in \mathcal{A}$ representing the sequence of chosen policies and $0
\triangleq \tau_0 < \tau_1 < \dots \le T$ representing the times of policy changes (from
now on termed \emph{switching times}). %We also assume $\xi_k \neq \xi_{k+1}$, $k \in \mathbb{N}_+$.
We denote the entire strategy by the right-continuous piecewise constant process $\xi \colon
[0,T] \times \Omega \to \mathcal{A}$, with $\xi_t = \xi_k$ if $ \tau_{k} \leqslant t <
\tau_{k+1}$ or
\begin{align}\label{eq:xi}
\xi_t = \sum_{\tau_{k+1} \leq  T} \xi_k \cdot 1_{[\tau_k, \tau_{k+1})}(t).
\end{align}

Beyond running benefits, the controller also faces switching costs in changing her policy which
lead to inertia and hysteresis. If at time $t$, the controller changes her policy from $a$ to
$b$ and $M_t = i$ then an immediate cost $K_i(a,b) $ is incurred. The overall objective of the
controller is to maximize the total present value of all tracking benefits minus the switching
costs which is given by
\begin{align*}
\int_0^{T} \e^{- \rho t }  \left( \sum_{i \in E }  c_{i}(\xi_t) 1_{ \{ M_{t} = i  \} } \right)
\,dt - \sum_k \e^{-\rho \tau_k} \left(\sum_{i \in E} K_i(\xi_{\tau_k-}, \xi_{\tau_k})\cdot 1_{
\{ M_{\tau_k} = i \} } \right),
\end{align*}
where  $\rho \ge 0$ is the discount factor.

Since $M$ is unobserved, the controller must carry out a filtering procedure. We postulate that
she collects information about $M$ via a Bayesian framework. Let $\vp = (\pi_1, \ldots, \pi_m)
\triangleq \left( \P\{M_0 = 1\}, \ldots, \P\{M_0 = m\} \right)$ be the initial (prior) beliefs
of the controller about $M$ and $\P^{\vp}$ the corresponding conditional probability law. The
controller starts with beliefs $\pi$, observes $X$, updates her beliefs and adjusts her policy
accordingly. Because only $X$ is observable, the strategy $\xi$ should be determined by the
information generated by $X$, namely each $\tau_k$ must be a stopping time of the filtration
$\mathcal{F}^X$ of $X$. Similarly, the value of each $\xi_k$ is determined by the information
$\mathcal{F}^X_{\tau_k}$ revealed by $X$ until $\tau_k$. These notions and the precise updating
mechanism will be formalized in Section \ref{sec:vP}. We denote by $\U(T)$ the set of all such
\emph{admissible strategies} on a time interval $[0,T]$. Since strategies with infinitely many
switches would have infinite costs, we exclude them from $\U(T)$.

Starting with initial policy $a \in \A$ and beliefs $\vp$, the performance of a given policy
$\xi \in \U(T)$ is
\begin{equation}
J^{\xi}(T,\vp,a) \triangleq \E^{\vp,a} \left[  \int_0^{T} \e^{- \rho t }   \left( \sum_{i \in
E} c_i(\xi_t) 1_{ \{  M_{t} = i \} } \right) dt - \sum_{k} \e^{- \rho \tau_k} \!\left(\sum_{i
\in E} K_i(\xi_{k-1}, \xi_{k}) \cdot 1_{ \{ M_{\tau_k} = i \} } \right) \right].
\end{equation}
The first argument in $J^\xi$ is the remaining time to maturity. The optimization problem is to
compute
\begin{align}
\label{def:U} U (T, \vp, a) \triangleq \sup_{ \xi \in \U(T)} J^{\xi}(T,\vp,a),
\end{align}
and, if it exists, find an admissible strategy $\xi^*$ attaining this value. In this paper we
solve \eqref{def:U}, including giving a full characterization of an optimal control $\xi^*$ and
a deterministic numerical method for computing $U$ to arbitrary level of precision. The
solution will proceed in two steps: an initial filtering step and a second optimization step.
The inference step is studied in Section \ref{sec:probStat}, where we convert the optimal
control problem with partial information \eqref{def:U}  into an equivalent \emph{fully
observed} problem in terms of the a posteriori probability process $\vP$.  The process $\vP$
summarizes the dynamic updating of controller's beliefs about the Markov chain $M$ given her
point process observations. The explicit dynamics of $\vP$ are derived in Proposition
\ref{cor:pdp}, so that the filtering step is completely solved. The main part of the paper then
analyzes the resulting optimal switching problem \eqref{def:V} in Sections \ref{sec:sequential}
and \ref{sec:opt-strategy}.

To our knowledge, the finite horizon partially observed switching control problem (which might
be viewed as an impulse control problem in terms of $\xi$) defined in \eqref{def:U}, has not
been studied before. However, it is closely related to optimal stopping problems with partially
observable Cox processes that have been extensively looked at starting with the Poisson
Disorder problems, see e.g.\ \cite{PeskirShiryaev, MR2003i:60071,BD03,bdk05,BS06}. In
particular, \cite{BS06} solved the Poisson disorder problem when the change time has phase type
prior distribution by showing that it is equivalent to an optimal stopping problem for a hidden
Markov process (which has several transient states and one absorbing state) that is indirectly
observed through a point process. Later \cite{LS07} solved a similar optimal stopping problem
in which all the states of the hidden Markov chain are recurrent. Both of these works can be
viewed as a special case of \eqref{def:U}, see Remark \ref{rem:opt-stopping-as-special-case}.
Our model can also be viewed as the continuous-time counterpart of discrete-time sequential
$M$-ary detection in hidden Markov models, a topic extensively studied in sequential analysis,
see e.g.\ \cite{TartakovskyEtal06,Aggoun03}.

Filtering problems with point process observations is a well-studied area; let us mention the
work of \cite{Arjas92}, \cite{CeciGerardi98} and the reference volume \cite{ElliottBook}. In
our model we use the previous results obtained in \cite{BS06,LS07} to derive an explicit
filter; this allows us then to focus on the separated fully-observed optimal switching problem
using the new hyper-state. Let us also mention the recent paper of \cite{ChopinVarini07} who
study a simulation-based method for filtering in a related model, but where an explicit filter
is unavailable and must be numerically approximated.

The techniques that we use to solve the optimal switching/impulse control problem are different
from the ones used in the continuous-time optimal control problems mentioned above. The main
tool in solving the optimal stopping problems (in the multi-dimensional case, the tools in the
one dimensional case are not restricted to the one described here) is the approximating
sequence that is constructed by restricting the time horizon to be less than the time of the
$n$-th observation/jump of the observed point process. This sequence converges to the value
function uniformly and exponentially fast. However, in the impulse control problem, the
corresponding approximating sequence is constructed by restricting the sum of the number of
jumps \emph{and} interventions to be less than $n$. This sequence converges to the value
function, however the uniform convergence in both $T$ and $\vp$ is not identifiable using the
same techniques.

As in \cite{CostaDavis89} and \cite{CostaRaymundo00} (also see \cite{MazziottoEtal88} for
general theory of impulse control of partially observed stochastic systems), we first
characterize the value function $U$ as the smallest fixed point of two functional operators and
obtain the aforementioned approximating sequence. Using one of these characterization results
and the path properties of the a posteriori probability process we obtain one of our main
contributions: the regularity of the value function $U$. We show that $U$ is convex in $\vp$,
Lipschitz in the same variable on the closure of its domain, and Lipschitz in the $T$ variable
uniformly in $\vp$. Our regularity analysis leads to the proof of the continuity of $U$ in both
$T$ and $\vp$ which in turn lets us explicitly describe an optimal strategy.

The other characterization of $U$ as a fixed point of the first jump operator is used to
numerically implement the optimal solution and find the value function. In general, very little
is known about numerics for continuous-time control of general hidden Markov models, and this
implementation is another one of our contributions. We combine the explicit filtering equations
together with special properties of piecewise deterministic processes \citep{davis93} and the
structure of general optimal switching problems to give a complete computational scheme. Our
method relies only on deterministic optimization sub-problems and lets us avoid having to deal
with first order quasi-variational inequalities with integral terms that appear in related
stochastic control formulations (see remark \ref{rem:qvi} below). We illustrate our approach
with several examples on a finite/infinite horizon and a hidden Markov chain with two or three
states.

Our framework has wide-ranging applications in operations research, management science and
applied probability. Specific cases are discussed in the next subsection. As these examples
demonstrate, our approach leads to sensible policy advice in many scenarios. Most of the
relevant applied literature treats discrete-time stationary problems, and our model can be seen
as a finite-horizon, continuous-time generalization of these approaches.

The rest of the paper is organized as follows: In Section~\ref{sec:apps} we propose some
applications of our modeling framework.
 In Section \ref{sec:probStat} we describe an equivalent fully observed problem in terms of the a posteriori probability process $\vP$. We also analyze the dynamics of $\vP$. In Section \ref{sec:sequential} we show that $U$ satisfies two different dynamic
programming equations. The results of Section~\ref{sec:sequential} along with the path
description of $\vP$ allows us to study the regularity properties of $U$ and describe an
optimal strategy in Section~\ref{sec:opt-strategy}. Our model can be extended beyond
\eqref{def:U}, in particular to cover the case of infinite horizon and the case in which the
costs are incurred at arrival times. The extensions are described in  Section~\ref{sec:extend}.
Extensive numerical analysis of several illustrative examples is carried out in
Section~\ref{sec:examples}.

\subsection{Applications}\label{sec:apps} % Combine with numerical examples
In this section we discuss case studies of our model and the relevant applied literature.

% As a specific example...

\subsubsection{Cyclical Economic Policy Making}
The economic business cycle is a basis of many policy making decisions. For instance, the
country's central bank attempts to match its monetary policy, so as to have low interest rates
in periods of economic recession and high interest rates when the economy overheats. Similarly,
individual firms will time their expenditures to coincide with boom times and will cut back on
capital spending in unfavorable economy states. Finally, investors hope to invest in the bull
market and stay on the sidelines during the bear market. In all these cases, the precise
current economy state is never known. Instead, the agents collect information via economic
events, surveys and news, and act based on their dynamic beliefs about the environment.
Typically, such news consist of discrete events (e.g.\ earnings pre-announcements,
geo-political news, economic polls) which cause \emph{instantaneous} jumps in agents' beliefs.
Thus, it is natural to model the respective information structure by observations of a
modulated compound Poisson process. Accordingly, let $M$ represent the current state of the
economy and let the observation $X$ correspond to economic news. Inability to correctly
identify $M$ will lead to (opportunity) costs $c_{M_s}(\xi_s)$. Hence, one may take $\A = E$
and $c_a(a) =0, c_a(b) < 0$. The strategy $\xi$ represents the set of possible actions of the
agent. The switching costs of the form $K(\xi_s, \xi_{s-}) > 0$ correspond to the costly
influence of the Federal Reserve changing its interest rate policy, or to the transaction costs
incurred by the investor who gets in/out of the market. Depending on the particular setting,
one may study this problem both in finite- and infinite-horizon setting, and with or without
discounting. For instance, a firm planning its capital budgeting expenses might have a fixed
horizon of one year, while a central bank has infinite horizon but discounts future costs. A
corresponding numerical example is presented in Section \ref{sec:fed-target}.

%
%A related classical model concerns optimal replacement policies in deteriorating Markov
%systems, see \cite{Ross71,MakisJiang,AggounBenkherouf02}.

\subsubsection{Matching Regime-Switching Demand Levels}
Many customer-oriented businesses experience stochastically fluctuating demand. Thus, internet
servers face heavy/light traffic; manufacturing managers observe cyclical demand levels;
customer service centers have varying frequencies of calls. Such systems can be modeled in
terms of a compound Poisson request process $X$ modulated by the partially known system state
$M$. Here, $X$ serves the \emph{dual} role of representing the actual demands and conveying
information about $M$. The objective of the agent is to dynamically choose her strategy $\xi$,
so as to track current demand level. For instance, an internet server receives asynchronous
requests $Y_\ell$, $\ell = 1,2,\ldots$ (corresponding to jumps of $X$) that take
$c(Y_\ell,\xi_t)$ time units to fulfill. The rate of requests and their complexity distribution
depend on $M$. In turn, the server manager can control how much processing power is devoted to
the server: more processors cut down individual service times but lead to higher fixed
overhead. Such a model effectively corresponds to a controlled $M(\lambda)/G/\infty$-queue,
where the arrival rate $\lambda$ is $M$-modulated, and where the distribution of service times
depends both on $M$ and the control $\xi$. A related computational example concerning a
customer call center is treated in Section \ref{sec:customer-call}.

A concrete example that has been recently studied in the literature is the insurance premium
problem. Insurance companies handle claims in exchange for policy premiums. A standard model
asserts that claims $Y_1, Y_2, \ldots$ form a compound (time-inhomogeneous) Poisson process
$X$. Suppose that the rate of claims is driven by some state variable $M$ that measures the
current background risk (e.g.\ climate, health epidemics, etc.), with the latter being
unobserved directly. In \cite{Aggoun03}, such a model was studied (in discrete time) from the
inference point of view, deriving the optimal filter for the insurance environment $M$ given
the claim process. Assume now that the company can control its continuous premium rate
$c^2(\xi_t)$, as well as its deductible level $c^1(\xi_t)$. High deductibles require lowering
the premium rate, and are therefore only optimal in high-risk environments. Furthermore,
changes to policy provisions (which has a finite expiration date $T$) are costly and should be
undertaken infrequently. The overall objective is thus,
$$ \sup_{\xi \in \U(T)} \E^{\vp,a} \left[
-\sum_{j=1}^{N(T)} \e^{-\rho \sigma_j} (Y_j-c^1(\xi_{\sigma_j}))_+ + \int_0^T c^2(\xi_t) \, dt
- \sum_k \e^{- \rho \tau_k} \left(\sum_{i \in E} K_i(\xi_{k-1}, \xi_{k}) \cdot 1_{ \{
M_{\tau_k} = i \} } \right)   \right],
$$
where $N$ is the counting process for the number of claims. The resulting cost structure, which
is a variant of \eqref{def:U}, is described in Section \ref{sec:discrete-costs}. %A related case
%of inference for a partially observed inventory process is analyzed in
%\cite{AggounBenmerzouga}.

\subsubsection{Security Monitoring}
Classical models of security surveillance (radar, video cameras, communication network monitor)
involve an unobserved system state $M$ representing current security (e.g.\ $E = \{ 0, 1\}$,
where $0$ corresponds to a `normal' state and $1$ represents a security breach) and a signal
$X$. The signal $X$ records discrete events, namely artifacts in the surveyed space (radar
alarms, camera movement, etc.). Benign artifacts are possible, but the intensity $\lambda$ of
$X$ increases when $M_t = 1$. If the signal can be decomposed into further sub-types, then $X$
becomes a marked point process with marks $(Y_\ell)$. The goal of the monitor is to correctly
identify and respond to security breaches, while minimizing false alarms and untreated security
violations. Classical formulations \citep{TartakovskyEtal06,PeskirShiryaev} only analyze
optimality of the first detection. However, in most practical problems the detection is ongoing
and discrete announcement costs require studying the entire (infinite) sequence of detection
decisions. Accordingly, our optimal switching framework of \eqref{def:U} is more appropriate.

As a simplest case, the monitor can either declare the system to be sound $\xi_t=1$, or declare
a state of alarm $\xi_t=2$. This produces $M$-dependent penalty costs at rate $\sum_{j\in E}
c_j(\xi_t) 1_{\{ M_t = j\}} dt$; also changing the monitor state is costly and leads to costs
$K$. A typical security system is run on an infinite loop and one wishes to minimize total
discounted costs, where the discounting parameter $\rho$ models the effective time-horizon of
the controller (i.e.\ the trade-off between the myopically optimal announcement and long-run
costs). Such an example is presented in Section \ref{sec:security-detection}.

\subsubsection{Sequential Poisson Disorder Problems}
Our model can also serve as a generalization of Poisson disorder problems,
\citep{bdk05,MR2003i:60071}. Consider a simple Poisson process $X$ whose intensity $\lambda$
sequentially alternates between $\lambda_0$ and $\lambda_1$. The goal of the observer is to
correctly identify the current intensity; doing so produces a running reward at rate
$c_0(\xi_t)$ per unit time, otherwise a cost at rate $c_1(\xi_t)$ is assessed, where $\xi$ is
the control process. Whenever the observer changes her announcement, a fixed cost $K$ is
charged in order to make sure that the agent does not vacillate. Letting $M$, $M_t \in \{0,1\}$
denote the intensity state, and $\lambda = \lambda_{M_t}$ this  example yet again fits into the
framework of \eqref{def:U}. Obvious generalizations to multiple values of $\lambda$ and
multiple announcement options for the observer can be considered. Again, one may study the
classical infinite-horizon problem, or the harder time-inhomogeneous model on finite-horizon,
where the observer must also take into account time-decay costs.

% Put here the two-state example: finite horizon, no discounting

% average costs, compound Poisson

\section{Problem Statement}\label{sec:probStat}
In this section we rigorously define the problem statement and show that it is equivalent to a
fully observed impulse control problem using the conditional probability process $\vP$. We then
derive explicitly the dynamics of $\vP$. First, however we give a construction of the
probability measure $\P$ and the formal description of $X$.

\subsection{Observation Process}
Let $(\Omega, \mathcal{H}, \P_0)$ be a probability space hosting two independent elements: (i)
a continuous time Markov process $M$ taking values in a finite set $E$, and with infinitesimal
generator $Q=(q_{ij})_{i,j \in E }$, (ii) a compound Poisson process $X$ with intensity
$\lambda_1$ and jump size distribution $\nu_1$ on $\R^d$. Let $\mathbb{F}=\{\mathcal{F}^X_t\}$
be the natural filtration of $X$ enlarged by $\P_0$-null sets, and consider its initial
enlargement $\mathbb{G}=\{\G_t\}_{t \geq 0}$ with  $\mathcal{G}_t \triangleq
\sigma(\mathcal{F}^X_t , \sigma(\{M_t\}_{t \geq 0}))$ for all $t \geq 0$. The filtration
$\mathbb{G}$ summarizes the information flow of a genie that observes the entire path of $M$ at
time $t=0$.

Denote by $\sigma_0, \sigma_1 , \ldots$ the arrival times of the process $X$,
\begin{align*}
\sigma_\ell \triangleq \inf \{ t > \sigma_{\ell-1} : X_t \ne X_{t- }\}, \qquad \ell \ge 1
\qquad \text{with $\sigma_0 \equiv 0$.}
\end{align*}
and by $Y_1, Y_2 ,\ldots$ the $\R^d$-valued marks observed at these arrival times:
\begin{align*}
Y_\ell = X_{\sigma_\ell } - X_{\sigma_\ell- }, \qquad \ell \ge 1.
\end{align*}
Then in terms of the counting random measure
\begin{equation}\label{eq:p-crand}
p((0, t],A) \triangleq \sum_{\ell=1}^{\infty} 1_{\{\sigma_\ell \leq t\}}1_{\{Y_\ell \in A\}},
\end{equation}
where  $A$ is a Borel set in $\R^d$, we can write the observation process $X$ as
\begin{equation*}
X_t=X_0+\int_{(0,t] \times \mathbb{R}^d} y\, p(ds, dy).
\end{equation*}

Let us introduce the positive constants $\lambda_2, \ldots, \lambda_m$ and the distributions
$\nu_2, \ldots, \nu_m$. We also define the total measure $\nu \triangleq \nu_{1} + \ldots +
\nu_m $, and let $f_i(\cdot)$ be the density of $\nu_i$ with respect to $\nu$. Define
\[
R(t,y) \triangleq \frac{1}{\lambda_1 f_1(y)}\sum_{i \in E} 1_{\{M_t=i\}} \lambda_i f_{i}(y),
\qquad t \geq 0, \, y \in \R^d.
\]
and  denote the  $(\P_0,\mathbb{G})-$ (or $(\P_0,\F)$)-compensator of $p$ by
\begin{equation} p_0((0,t] \times A)= \lambda_1 t \int_{A} f_1(y)
\nu(dy), \qquad t \geq 0, \, A \in \mathcal{B}(\R^d).
\end{equation}
We will use $R(t,y)$ and $p_0$ to change the underlying probability measure to a new
probability measure $\P$ on $(\Omega, \mathcal{H})$ defined by
\begin{equation*}
\frac{d \P}{d \P_0} \bigg|_{\G_t}=Z_t,
\end{equation*}
where the stochastic exponential $Z$ given by
\begin{equation*}
Z_t \triangleq \exp\left\{\int_{(0,t] \times \R^d} \log (R(s,y)) \, p(ds,dy)-\int_{(0,t] \times
\R^d}[R(s,y)-1] \, p_0(ds,dy)\right\},
\end{equation*}
is a $(\P_0,\mathbb{G})$-martingale. Note that $\P$ and $\P_0$ coincide on $\G_0$ since
$Z_0=1$, therefore law of the Markov chain $M$ is the same under both probability measures.
Moreover, the $(\P,\mathbb{G})$-compensator of $p$ becomes
\begin{equation}
p_1((0,t],A)=  \sum_{i \in E} \int_{(0,t]}\! 1_{\{M_s=i\}}\lambda_i \int_A f_{i}(y)\nu(dy)\,ds.
\end{equation}
see e.g.\ \cite{JS}. The last statement is equivalent to saying that under this new
probability, $X$ has the form
\begin{align}
\label{def:X} X_t \triangleq X_0 + \int_0^t  \sum_{ i \in E} 1_{ \{ M_s =i \} } \, dX^{(i)}_s ,
\quad t\ge 0,
\end{align}
in which $X^{(1)} , \ldots, X^{(m)} $ are  independent compound Poisson processes with
intensities and jump size distributions  $ (\lambda_1, \nu_1), \ldots , (\lambda_m , \nu_m)$,
respectively. Such a process $X $ is called a Markov-modulated Poisson process
\citep{KarlinTaylor2}. By construction, the observation process $X$ has independent increments
conditioned on $M = \{ M_t\}_{t \ge 0}$. Thus, conditioned on $\{M_{\sigma_\ell} = i\}$, the
distribution of $Y_\ell$ is $\nu_i(\cdot)$ on $( \R^d, \mathcal{B}(\R^d)) $.

\subsection{Equivalent Fully Observed Problem.}
Let $D \triangleq \{ \vp \in [0,1]^m \colon \pi_1 + \ldots + \pi_m =1 \}$ be the space of prior
distributions of the Markov process $M$. Also, let $\s(s) = \{ \tau \colon
\F-\text{stopping time}, \tau \le s, \P-\text{a.s} \}$ denote the set of all $\F$-stopping
times smaller than or equal to $s$.

We define the $D$-valued \emph{conditional probability process} $\vP(t) \triangleq \left(
\Pi_1(t), \ldots , \Pi_m(t) \right)$ such that
\begin{align}
\label{def:Pi-i}
 \Pi_i(t) = \P \{  M_t =i | \Fc^X_t \}, \quad \text{for $i \in E$, and $t \ge 0$}.
\end{align}
Each component of $\vP$ gives the conditional probability that the current state of $M$ is $\{
i\}$ given the information generated by $X$ until the current time $t$. Using the process $\vP$
we now convert \eqref{def:U} into a standard optimal stopping problem.

\begin{prop}
\label{prop:new-V}
The performance of a given strategy $\xi \in \U(T)$ can be written as
\begin{align}
\label{def:V}
 J^{\xi}(T,\vp,a) =\E^{\vp,a} \left[   \int_0^{T} \e^{- \rho t }  \, C(\vP(t), \xi_t) \, dt
 - \sum_k \e^{- \rho \tau_k} K(\xi_{\tau_{k}-},\xi_{\tau_k},\vP({\tau_k}))   \right] ,
\end{align}
in terms of the functions
\begin{align}
\label{def:C} C(\vp, a) \triangleq  \sum_{i \in E }  c_i(a) \pi_i, \qquad\text{and}\qquad
K(a,b,\vp) \triangleq \sum_{i \in E} K_i(a,b) \pi_i.
\end{align}
\end{prop}

Proposition~\ref{prop:new-V} above states that solving the problem in \eqref{def:U} is
equivalent to solving an impulse control problem with state variables $\vP$ and $\xi$. As a
result, the filtering and optimization steps are completely separated. In our context with
optimal switching control, the proof of this separation principle is immediate (see e.g. 
\cite[pp. 166-167]{Sh78}).
In more general problems with continuous controls, the result is more delicate, see
\cite{CeciGerardi98}.

We proceed to discuss the technical assumptions on $C$ and $K$. Note that by construction
$C(\cdot,a)$ and $K(a,b,\cdot)$ are linear. Moreover, $C$ is bounded since $E$ is finite, so
there is a constant denoted $c = \max_{i \in E} |c_i|$ that uniformly bounds possible rates of
profit, $|C(\vp,a)| \le c$. For the switching costs $K$ we assume that they satisfy the
triangle inequality
\begin{equation*}
K_i(a,b)+ K_i(b,c) \geq K_i(a,c), \quad\text{and}\quad K_i(a,b) > k_0 > 0 \quad \text{for}
\quad i \in E;\, a,b,c \in \A.
\end{equation*}
By the above assumptions on the switching costs and because possible rewards are uniformly
bounded, with probability one the controller only makes finitely many switches and she does not
make two switches at once. Without loss of generality we will also assume that every element in
$\xi \in \U(T)$ satisfies
\begin{equation}\label{eq:sum-expe-sw-cst}
\E^{\vp,a} \left[ \sum_k \e^{- \rho \tau_k} K(\xi_{\tau_{k}-},\xi_{\tau_k},\vP({\tau_k}))
\right]<\infty.
\end{equation}
Otherwise, the cost associated with a strategy $\xi$ would be $-\infty$ since
\begin{equation*}
\E^{\vp,a} \left[   \int_0^{T} \e^{- \rho t }  \, |C(\vP(t), \xi_t)| \, dt \right] \leq c \,T,
\end{equation*}
and taking no action would be better than applying $\xi$.

In the sequel we will also make use of the following auxiliary problems. First, let $U_0$ be
the value of no-action, i.e.,
\begin{equation}\label{def:U-0}
U_0(T,\vp,a)=\E^{\vp,a} \left[   \int_0^{T} \e^{- \rho t }  \, C(\vP_t, a) \, dt\right].
\end{equation}
Also in reference to (\ref{def:U}), we will consider the restricted problems
\begin{equation}\label{eq:aux}
U_n(T,\vp,a)\triangleq
 \sup_{ \xi \in \U_n(T)} J^{\xi}(T,\vp,a), \qquad n \ge 1,
\end{equation}
in which $\U_n(T)$ is a subset of $\U(T)$ which contains strategies with at most $n \ge 1$
interventions up to time $T$.

\subsection{Sample paths of $\vP$.}\label{sec:vP}
In this section we describe the filtering procedure of the controller, i.e.\ the evolution of
the conditional probability process $\vP$.  Proposition \ref{cor:pdp} explicitly shows that the
processes $\vP$ and $(\vP,\xi)$ are piecewise deterministic processes and hence have the strong
Markov property, \cite{davis93}. This description of paths of the conditional probability
process is also discussed in Proposition 2.1 in \cite{LS07} and Proposition 2.1 of
\cite{BS06}. We summarize the needed results below.

Let
\begin{align}\label{def:I-t}
I(t) \triangleq \int_0^t \sum_{i=1}^m \lambda_i 1_{\{M_s = i \}}\, ds,
\end{align}
so that the probability of no events for the next $u$ time units is $\P^{\vp} \{ \sigma_1
> u \} = \E^{\vp}[ \e^{-I(u)}]$. Then for $\sigma_\ell \le t \le t+ u < \sigma_{\ell+1}$, we have
\begin{align}
\label{eq:semi-group} \Pi_i (t+u) &= \frac{    \P^{\vp}  \{ \sigma_1 > u , M_u =i  \}  }
 {   \P^{\vp}  \{ \sigma_1 > u  \} }\Bigg|_{\vp = \vP(t)}.
\end{align}
On the other hand, upon an arrival of size $Y_\ell$, the conditional probability $\vP$
experiences a jump
\begin{align}\label{eq:jumps-of-vP}
\Pi_i (\sigma_{\ell+1} ) = \frac{ \lambda_i f_i(Y_{\ell+1}) \Pi_i (\sigma_{\ell+1} - ) } {
\sum_{j \in E} \lambda_j f_j(Y_{\ell+1})  \Pi_j (\sigma_{\ell+1} - ) }, \qquad \text{for }\ell
\in \N.
\end{align}

To simplify \eqref{eq:semi-group}, define $\vx (t, \vp) \equiv (x_1(t, \vp), \ldots , x_m(t,
\vp))$ via
\begin{align}
\label{eq:x-i} x_i(t, \vp) \triangleq \frac{    \P^{\vp}  \{ \sigma_1 > t , M_t =i  \}  }
 {   \P^{\vp}  \{ \sigma_1 > t  \} }
=\frac{  \E^{\vp} \left[ 1_{\{ M_{t} =i\}} \cdot \e^{ - I(t)}   \right]  } {\E^{ \vp }
\left[   \e^{ - I(t)}   \right] }
 , \qquad \text{for $i \in E$.}
\end{align}
It can be checked easily that the paths $t \mapsto \vx(t,\vp)$ have the semigroup property
$\vx(t+u , \vp ) = \vx (u , \vx(t,\vp) )$. In fact, $\vx$ can be described as a solution of
coupled first-order ordinary differential equations. To observe this fact first recall
\citep{DarrochMorris,Neuts,KarlinTaylor2} that the vector
\begin{align}\label{def:m}
\vec{m} (t , \vp ) \equiv ( m_1 (t , \vp ), \ldots , m_m (t , \vp ) ) \triangleq \Bigl( \,
\E^{\vp,a} \left[ 1_{\{ M_{t} =1\}} \cdot \e^{ - I(t)}   \right]  , \ldots ,
  \E^{\vp,a} \left[ 1_{\{ M_{t} =m\}} \cdot \e^{ - I(t)}   \right] \, \Bigr)
\end{align}
has the form %$\vec{m} (t , \vp ) = \vp \cdot \e^{ t( Q - \Lambda) }$,
\begin{align*}
\vec{m} (t , \vp ) = \vp \cdot \e^{ t( Q - \Lambda) },
\end{align*}
where $\Lambda$ is the $m \times m$ diagonal matrix with $\Lambda_{i,i}  = \lambda_i$. Thus,
the components of $\vec{m} (t , \vp )$ solve $ d m_i (t, \vp ) / dt  = - \lambda_i m_i (t,\vp )
+ \sum_{j \in E}  m_j (t,\vp) \cdot q_{j, i}$ and together with the chain rule and
(\ref{eq:x-i}) we obtain
\begin{align}
\label{eq:vx-dyn}
\frac{d x_i (t,\vp) }{dt} = \left( \sum_{j\in E} q_{j,i}  x_j (t,\vp) - \lambda_i  x_i (t,\vp)  + x_i (t,\vp)   \sum_{j \in E} \lambda_{j}  x_j (t,\vp)  \right). %+ \int_{\R_+ } \frac{\Pi_i (t) \lambda_i p_i (y) }{ \sum_{ j \in E}  \Pi_j (t) \lambda_j p_j (y) }   N(dt,dy)
\end{align}
For the sequel we note again that $\P^{\vp} \left\{ \sigma_1 \in du , M_u =i \right\} =
\E^{\vp,a} \left[ \lambda_i 1_{ \{ M_u =i \} } \e^{- I(u)} \right] du =\lambda_i \, m_i(u,\vp)
\, du $.

The preceding equations \eqref{eq:semi-group} and \eqref{eq:jumps-of-vP} imply that
\begin{prop} \label{cor:pdp} The
process $\vP$ is a piecewise-deterministic, $(\P, \F)$-Markov process. The paths have the
characterization
     \begin{align}\label{eq:rel-pi-x} \left\{
     \begin{aligned}
    \vP(t)&=  \vx \left(t-\sigma_\ell,\vP({\sigma_\ell})\right), \qquad
\qquad \sigma_\ell \leq t< \sigma_{\ell+1}, \;\; \ell\in \mathbb{N} \qquad \\
\vP(\sigma_\ell)&= \left( \frac{  \lambda_1 f_1(Y_\ell) \Pi_1(\sigma_\ell-) }{ \sum_{j \in E}
\lambda_j f_j(Y_\ell) \Pi_j(\sigma_\ell-) }, \ldots \frac{  \lambda_m f_m(Y_\ell)
\Pi_m(\sigma_\ell-) }{ \sum_{j \in E} \lambda_j f_j(Y_\ell) \Pi_j(\sigma_\ell-) }\right)
\end{aligned} \right\}.
\end{align}
\end{prop}

Alternatively, we can describe $\vP$ in terms of the random measure $p$,
\begin{align}\notag
d \Pi_i(t) & = \mu_i( \vP({t-}) )\, dt + \int_{\R^d} J_i(\vP({t-}),y) \,
p(dt,dy),
\end{align}
for all $i \in E$, where
\begin{equation}\label{eq:mu-J-operators}
\begin{split}
 \mu_i(\vp)=\sum_{j \in E}q_{j,i} \pi_j+ \lambda \pi_{i} \left(\sum_{j \in E} \lambda_j \pi_j-\lambda_i\!\right), \quad\text{and}\quad J_i(\vp,y)= \pi_i \cdot \left(\frac{\lambda_i f_i(y)}{\sum_{j \in E} \lambda_j f_j(y)
\pi_j}-1\right).
\end{split}
\end{equation}
Here, one should also note that the $(\P,\F)$-compensator of the random measure $p$ is
\begin{equation*}
\tilde{p}((0,t] \times A)=  \sum_{j \in E} \int_0^t \int_A \lambda_j f_j(y) \Pi_j(s) \, dy\,ds,
\qquad t \geq 0, A \, \text{ Borel.}
\end{equation*}

In more general models with point process observations, an explicit filter for $\vP$ would not
be available and one would have to resort to simulation-based approaches, see e.g.\
\cite{ChopinVarini07}. The subsequent optimization step would then appear to be intractable,
though an integrated Markov chain Monte Carlo paradigm for filtering and optimization was
proposed in \cite{MullerEtal04}.

\section{Two Dynamic Programming Equations for the Value Function}
\label{sec:sequential} In this section we establish two dynamic programming equations for the
value function $U$. The first key equation  \eqref{eq:dyn-prog} reduces the solution of the
problem \eqref{def:U} to studying a system of coupled optimal stopping problems. The second
dynamic programming principle of Proposition \ref{eq:hat-wn} shows that the value function is
also the fixed point of a first jump operator. The latter representation will be useful in the
numerical computations.

\subsection{Coupled Optimal Stopping Operator}
In this section we show that $U$ solves a coupled optimal stopping problem. Combined with
regularity results in Section \ref{sec:opt-strategy}, this leads to a direct characterization
of an optimal strategy. The analysis of this section parallels the general framework of impulse
control of piecewise deterministic processes (pdp) developed by
\cite{CostaDavis89,LenhartLiao}. It is also related to optimal stopping of pdp's studied in
\cite{gugerli,CostaDavis88}. %We refer to \cite{MazziottoEtal88} for
%general theory of impulse control with partial information, and \cite{PeskirShiryaev, MR2003i:60071, BD03, bdk05, BS06, LS07} for optimal stopping problems with partial information.

Let us introduce a functional operator $\M$ whose action on a test function $w$ is
\begin{align}\label{eq:intervention-op} \M w(T,\vp,a) \triangleq \max_{b \in \A, \, b \neq
a} \Bigl\{w(T, \vp, b) - K(a,b,\vp) \Bigr\}.
\end{align}
The operator $\M$ is called the \emph{intervention} operator and denotes the maximum value that
can be achieved if an immediate best change is carried to the current policy. Assuming some
ordering on the finite policy set $\A$, let us denote the smallest policy choice achieving the
maximum in \eqref{eq:intervention-op} as
\begin{align}\label{eq:optimal-action}
d_{\M w}(T,\vp, a) \triangleq \min_{b \in \A} \Bigl\{ w(T,\vp, b) - K(a,b,\vp) = \M w(T,\vp,a)
\Bigr\}.
\end{align}

The main object of study in this section is another functional operator $\G$ whose action is
described by the following optimal stopping problem:
\begin{equation}\label{def:G}
\G V(T, \vp,a)=\sup_{\tau \in \s(T)} \E^{\vp,a} \left[ \int_0^\tau \e^{-\rho s} C(\vP_s, a) \,
ds + \e^{-\rho \tau} \M V(T-\tau, \vP_{\tau}, a) \right],
\end{equation}
for $\,T \in \R_+,\, \vp \in D$, and $a \in \A$. We set $V_0 \triangleq U_0$ from
\eqref{def:U-0} and iterating $\G$ obtain the following sequence of functions:
\begin{equation}\label{defn:Un}
V_{n+1} \triangleq \G V_n, \quad n \geq 0.
\end{equation}
\begin{lemm}\label{lem:incr}
$(V_n)_{n \in \mathbb{N}}$ is an increasing sequence of functions.
\end{lemm}
\noindent In Section \ref{sec:opt-strategy} we will further show that $(V_n)$ are convex and
continuous.
\begin{proof}
The statement follows since
\begin{align*}
V_1(T,\vp,a) = \G V_0(T,\vp,a) & = \sup_{\tau \in \s(T)} \E^{\vp,a} \left[ \int_0^\tau
\e^{-\rho s} C(\vP_s, a) \, ds + \e^{-\rho \tau} \M V_0(T-\tau, \vP_{\tau}, a) \right]\\
& \ge \E^{\vp,a} \left[ \int_0^T \e^{-\rho s} C(\vP_s, a) \, ds \right] = U_0(T,\vp,a) =
V_0(T,\vp,a),
\end{align*}
and since $\G$ is a monotone/positive operator, i.e.\ for any two functions $f_1 \leq f_2$ we
have  $\G f_1 \leq \G f_2$, and
\end{proof}
The following proposition shows that the value functions $(U_n)_{n \in \mathbb{N}}$ of
\eqref{eq:aux}, which correspond to the restricted control problems over $\U_n(T)$, can be
alternatively obtained via the sequence of iterated optimal stopping problems in
\eqref{defn:Un}.
\begin{prop}\label{prop:UneqVn}
$U_n=V_n$ for $n \in \mathbb{N}$.
\end{prop}
\begin{proof}
By definition we have that $U_0=V_0$. Let us assume that $U_n=V_n$ and show that
$U_{n+1}=V_{n+1}$. We will carry out the proof in two steps.

\textbf{Step 1}. First we will show that  $U_{n+1} \leq V_{n+1}$. Let $\xi \in \U_{n+1}(T)$,
\begin{align*}
\xi_t = \sum_{k=0}^{n+1} \xi_k \cdot 1_{[\tau_k, \tau_{k+1})}(t), \quad t \in [0,T],
\end{align*}
with $\tau_0=0$ and $\tau_{n+1}=T$, be $\eps$-optimal for the problem in (\ref{eq:aux}), i.e.,
\begin{equation}\label{eq:eps-opt}
U_{n+1}(T,\vp,a)-\eps \leq J^{\xi}(T,\vp,a).
\end{equation}
Let $\tilde{\xi} \in \U_n(T)$ be defined as
\begin{equation*}
\tilde{\xi}_t=\sum_{k=0}^{n} \tilde{\xi}_k \cdot 1_{[\tilde{\tau}_k, \tilde{\tau}_{k+1})}(t), \quad t \in [0,T],
\end{equation*}
in which $\tilde{\tau}_0=0$, $\tilde{\xi}_0=a$, and $\tilde{\tau}_n=\tau_{n+1}$ ,
$\tilde{\xi}_n=\xi_{n+1}$,  for $n \in \mathbb{N}_+$. Using the strong Markov property of
$(\vP,\xi)$, we can write $J^{\xi}$ as
\begin{equation}\label{eq:eps-argument}
\begin{split}
J^{\xi}(T,\vp,a)&= \E^{\vp,a}\left[\int_0^{\tau_1}\e^{-\rho s} C(\vP_s,a)\, ds+\e^{-\rho
\tau_1}\left(J^{\tilde{\xi}}(T-\tau_1,\vP_{\tau_1},\xi_1)-K(a,\xi_1,\vP_{\tau_1})\right)\right]
\\ &\leq \E^{\vp,a}\left[\int_0^{\tau_1}\e^{-\rho s} C(\vP_s,a) \, ds+\e^{-\rho \tau_1}\left(V_{n}(T-\tau_1,\vP_{\tau_1},\xi_1)-K(a,\xi_1,\vP_{\tau_1})\right)\right]
\\ &\leq \E^{\vp,a}\left[\int_0^{\tau_1}\e^{-\rho s} C(\vP_s,a) \, ds+\
\e^{-\rho \tau_1} \M V_n  (T-\tau_1, \vP_{\tau_1},\xi_1)\right]
\\ &\leq \G V_{n}(T,\vp,a)=V_{n+1}(T,\vp,a).
\end{split}
\end{equation}
Here, the first inequality follows from induction hypothesis, the second inequality follows
from the definition of $\M$, and the last inequality from the definition of $\G$. As a result
of (\ref{eq:eps-opt}) and (\ref{eq:eps-argument}) we have that $U_{n+1} \leq V_{n+1}$ since
$\eps>0$ is arbitrary.

\textbf{Step 2}. To show the opposite inequality $U_{n+1} \geq V_{n+1}$, we will construct a
special $\bar{\xi} \in \U_{n+1}(T)$. To this end let us introduce
\begin{align}\left\{ \begin{aligned}
\bar{\tau}_1 & = \inf\{t\geq 0: \M V_n(T-t,\vP_t, a) \geq V_{n+1}(T-t,\vP_t,a)-\eps \}, \\
\bar{\xi}_1 & = d_{\M V_n}(T-\bar{\tau}_1,\vP_{\bar{\tau}_1},a). \end{aligned}\right.
\end{align}
\noindent  Let $\hat{\xi}_t = \sum_{k=0}^{n} \hat{\xi}_k \cdot 1_{[\hat{\tau}_k,
\hat{\tau}_{k+1})}(t)$, $\hat{\xi} \in \U_n(T)$ be $\eps$-optimal for the problem in which $n$
interventions are allowed, i.e.\ (\ref{eq:aux}). Using $\hat{\xi}$ we now complete the
description of the control $\bar{\xi} \in \U_{n+1}(T)$ by assigning,
\begin{equation}
\bar{\tau}_{n+1}=\hat{\tau}_n \circ \theta_{\tau_1}, \quad \bar{\xi}_{n+1}=\hat{\xi}_{n} \circ \theta_{\bar{\tau}_1}, \quad n \in \mathbb{N}_+,
\end{equation}
in which $\theta$ is the classical \emph{shift operator} used in the theory of Markov processes.

Note that $\bar{\tau}_1$ is an $\eps$-optimal stopping time for the stopping problem in the
definition of $\G V_n$. This follows from the classical optimal stopping theory since the
process $\vP$ has the strong Markov property. Therefore,
\begin{equation}\label{eq:tos}
\begin{split}
V_{n+1}(T,\vp,a)-\eps &\leq \E^{\vp,a}\left[\int_0^{\bar{\tau}_1}\e^{-\rho s}C(\vP_s,a) \,
ds+\e^{-\rho \bar{\tau}_1}\M V_{n}(T-\bar{\tau}_1,\vP_{\bar{\tau}_1},a)\right]
\\& \leq  \E^{\vp,a}\left[\int_0^{\bar{\tau}_1}\e^{-\rho s}C(\vP_s,a)\, ds+\e^{-\rho \bar{\tau}_1} \left(U_{n}(T-\bar{\tau}_1,\vP_{\bar{\tau}_1},\bar{\xi}_1)-K\left(a,\bar{\xi}_1,\vP_{\bar{\tau}_1}\right)\right)\right],
\end{split}
\end{equation}
in which the second inequality follows from the definition of $\bar{\xi}_1$ and the induction hypothesis.
It follows from (\ref{eq:tos}) and the strong Markov property of $(\vP,\xi)$ that
\begin{equation}
\begin{split}
V_{n+1}(T,\vp,a)-2\eps &\leq \E^{\vp,a}\left[\int_0^{\bar{\tau}_1}\e^{-\rho
s}C(\vP_s,a)\,ds+\e^{-\rho \bar{\tau}_1}
\left(U_{n}(T-\bar{\tau}_1,\vP_{\bar{\tau}_1},\bar{\xi}_1)-\eps-K\left(a,\xi_1,\vP_{\bar{\tau}_1}\right)\right)\right]
\\&\leq \E^{\vp,a}\left[\int_0^{\bar{\tau}_1}\e^{-\rho s}C(\vP_s,a)\, ds+\e^{-\rho \bar{\tau}_1} \left(J^{\hat{\xi}}(T-\bar{\tau}_1,\vP_{\bar{\tau}_1},\xi_1)-K\left(a,\xi_1,\vP_{\bar{\tau}_1}\right)\right)\right]
\\&=J^{\bar{\xi}}(T,\vp,a) \leq U_{n+1}(T,\vp,a).
\end{split}
\end{equation}
This completes the proof of the second step since $\eps>0$ is arbitrary.
\end{proof}
\begin{prop}\label{prop:vnconvU}
$\lim_{n \uparrow \infty}V_{n}(T,\vp,a)=U(T,\vp,a), \quad\text{for any}\quad T \in \R_+$, $\vp
\in D$, $a \in \A$.
\end{prop}
\begin{proof}
Fix $(T,\vp,a)$. The monotone limit $V(T,\vp,a)=\lim_{n \rightarrow \infty}V_n(T,\vp,a)$ exists
as a result of Lemma~\ref{lem:incr}. Since $\U_n(T) \subset \U(T)$, it follows that
$V_n(T,\vp,a)=U_n(T,\vp,a) \leq U(T, \vp, a)$. Therefore $V(T,\vp,a) \leq U(T,\vp,a)$. In the
remainder of the proof we will show that $V(T,\vp,a) \geq U(T,\vp,a)$.

Let $\xi \in \U(T)$ be given, and let $\tilde{\xi}_t :=\xi_{t \wedge \tau_n}$, $\tilde{\xi} \in
\U_n(T)$,  correspond to $\xi$ up to its $n$-th switch. Then
\begin{multline}\label{eq:int-step}
|J^{\xi}(T,\vp,a)-J^{\tilde{\xi}}(T,\vp,a)| \\ \leq \E^{\vp,a}\bigg[\int_{\tau_n}^{T}\e^{-\rho
s}|C(\vP_s,\xi_s)-C(\vP_s,\tilde{\xi}_{\tau_n})| \, ds + \sum_{k \geq n+1} \e^{-\rho
\tau_k}K(\xi_{\tau_{k-1}},\xi_{\tau_{k}},\vP_{\tau_{k}}) \bigg].
%\\& \leq 2c \, \E^{\vp,a}\left[\int_{\tau_n}^{T}\e^{-\rho s}ds\right]+\E^{\vp,a}\left[\sum_{k \geq n+1} \e^{-\rho \tau_k}K(\xi_{\tau_{k-1}},\xi_{\tau_{k}},\vP({\tau_{k}}))\right],
\end{multline}
\noindent Now, the right-hand-side of (\ref{eq:int-step}) converges to 0 as $n \rightarrow
\infty$: on the one hand observe that by monotone convergence theorem and
(\ref{eq:sum-expe-sw-cst})
\[
\lim_{n \rightarrow \infty}\E^{\vp,a}\left[\sum_{k \geq n+1} \e^{-\rho
\tau_k}K(\xi_{\tau_{k-1}},\xi_{\tau_{k}},\vP_{\tau_{k}})\right]=0.
\]
On the other hand, since there are only finitely many switches almost surely for any given
path,
\[
\lim_{n \rightarrow \infty}\int_{0}^{T}1_{\{s>\tau_n\}}\e^{-\rho
s}|C(\vP_s,\xi_s)-C(\vP_s,\tilde{\xi}_{\tau_n})| \, ds=0,
\]
and $\int_{\tau_n}^{T}\e^{-\rho s}|C(\vP_s,\xi_s)-C(\vP_s,\xi_{\tau_n})|ds \leq 2 c T$.
Therefore, the dominated convergence theorem  implies that
\[
\lim_{n \rightarrow \infty}\E^{\vp,a}\left[\int_{\tau_n}^{T}\e^{-\rho
s}|C(\vP_s,\xi_s)-C(\vP_s,\tilde{\xi}_{\tau_n})| \, ds\right]=0.
\]
As a result, for any $\eps > 0$ and $n$ large enough, we find
\begin{equation*}
|J^{\xi}(T,\vp,a)-J^{\tilde{\xi}}(T,\vp,a)| \leq \eps.
\end{equation*}
Now, since $\tilde{\xi} \in \U_n(T)$ we have $V_n(T,\vp,a) =U_{n}(T,\vp,a)\geq
J^{\tilde{\xi}}(T,\vp,a) \geq J^{\xi}(T,\vp,a)-\eps$ for sufficiently large $n$, and it follows
that
\begin{equation}\label{eq:WgewJ}
V(T,\vp,a)=\lim_{n \rightarrow \infty}V_{n}(T,\vp,a)\geq J^{\xi}(T,\vp,a)-\eps.
\end{equation}
Since $\xi$ and $\eps$ are arbitrary, we have the desired result.
\end{proof}

\begin{prop}\label{prop:dp}
The value function $U$ is the smallest solution of the dynamic programming equation $\G U=U$,
such that $U \geq U_0$. Thus,
\begin{align}\label{eq:dyn-prog}
U(T, \vp, a) = \sup_{\tau \in \s(T)} \E^{\vp,a} \left[ \int_0^\tau \e^{-\rho s} C(\vP_s, a) \,
ds + \e^{-\rho \tau} \M U(T-\tau, \vP_{\tau}, a) \right].
\end{align}
\end{prop}

\begin{proof}
\textbf{Step 1}. First we will show that $U$ is a fixed point of $\G$. Since $V_n \le U$,
monotonicity of $\G$ implies that
\begin{equation*}
V_{n+1}(T,\vp,a)\leq \sup_{\tau \in \mathcal{S}(T)}\E^{\vp,a}\left[\int_0^{\tau}\e^{-\rho
s}C(\vP_s,a)ds+ \e^{-\rho \tau}\M U(T-\tau, \vP_{\tau},a)\right].
\end{equation*}
Taking the limit of the left-hand-side with respect to $n$ and using Lemma~\ref{lem:incr} and
Proposition~\ref{prop:vnconvU} we have
\begin{equation*}
U(T,\vp,a) \leq \sup_{\tau \in \mathcal{S}(T)}\E^{\vp,a}\left[\int_0^{\tau}\e^{-\rho
s}C(\vP_s,a)\, ds+ \e^{-\rho \tau}\M U(T-\tau, \vP_{\tau},a)\right].
\end{equation*}
Let us obtain the reverse inequality. Let $\tilde{\tau} \in \mathcal{S}(T)$ be an
$\eps$-optimal stopping time for the problem in the definition of $\G U$, i.e.,
\begin{equation}\label{eq:eps-opti}
\begin{split}
\E^{\vp,a}&\left[\int_0^{\tilde{\tau}}\e^{-\rho s}C(\vP_s,a)\,ds+ \e^{-\rho \tilde{\tau}}\M
U(T-\tilde{\tau}, \vP_{\tilde{\tau}},a)\right]
\\& \geq \sup_{\tau \in \mathcal{S}(T)} \E^{\vp,a}\left[\int_0^{\tau}\e^{-\rho s}C(\vP_s,a)\,ds+ \e^{-\rho \tau}\M U(T-\tau, \vP_{\tau},a)\right]- \eps.
\end{split}
\end{equation}
Then, as a result of monotone convergence theorem and Proposition~\ref{prop:vnconvU}
\begin{equation}\label{eq:lim-argu}
\begin{split}
U(T,\vp,a) =\lim_{n \rightarrow \infty} V_{n}(T,\vp,a) & \geq \lim_{n \rightarrow
\infty}\E^{\vp,a}\left[\int_0^{\tilde{\tau}}\e^{-\rho s}C(\vP_s,a) \, ds+ \e^{-\rho
\tilde{\tau}}\M V_{n-1}(T-\tilde{\tau}, \vP_{\tilde{\tau}},a)\right]
\\&=\E^{\vp,a}\left[\int_0^{\tilde{\tau}}\e^{-\rho s}C(\vP_s,a) \, ds+ \e^{-\rho \tilde{\tau}}\M U(T-\tilde{\tau}, \vP_{\tilde{\tau}},a)\right].
\end{split}
\end{equation}
Now, (\ref{eq:eps-opti}) and (\ref{eq:lim-argu}) together yield the desired result since $\eps$ is arbitrary.

\textbf{Step 2}. Let $\tilde{U}$ be another fixed point of $\G$ satisfying $\tilde{U} \geq
U_0=V_0$. Then an induction argument shows that $\tilde{U}\geq U$: assume that $\tilde{U} \geq
V_n$. Then $\G \tilde{U} \geq \G V_n=V_{n+1}$, by the monotonicity of $\G$. Therefore for all
$n$, $\tilde{U} \geq V_n$, which implies that $\tilde{U} \geq \sup_n V_n =U$.

\end{proof}

To illustrate the nature of \eqref{eq:dyn-prog} consider the special case where $\A = \{1, 2
\}$ so that only two types of policies are available. In that case the intervention operator
$\M$ is trivial, $\M U(t,\vp, a) = U(t, \vp, 3-a)- K(a,3-a,\vp)$. For ease of notation we write
$U(t,\vp, 1) =: V(t,\vp)$, $U(t,\vp,2) =: W(t,\vp)$. It follows that \eqref{eq:dyn-prog} can be
written as two coupled optimal stopping problems:
\begin{align*}
\left\{ \begin{aligned} V(T, \vp) = \sup_{\tau \in \s(T)} \E^{\vp,a} \left[ \int_0^\tau
\e^{-\rho s} C(\vP_s, 1) \, ds + \e^{-\rho \tau} (W(T-\tau, \vP_{\tau})-K(1,2,\vp)) \right] \\
W(T, \vp) = \sup_{\tau \in \s(T)} \E^{\vp,a} \left[ \int_0^\tau \e^{-\rho s} C(\vP_s, 2) \, ds
+ \e^{-\rho \tau} (V(T-\tau, \vP_{\tau}) - K(2,1,\vp)) \right]. \end{aligned}\right.
\end{align*}
The next section discusses how to solve such coupled systems.

\begin{rem}
The value function $U(T,\cdot,a)$ is uniformly bounded. Indeed,
\begin{align*}
 U(T,\vp,a) & \ge U_0(T,\vp,a) = \E^{\vp,a} \left[ \int_0^T \e^{-\rho s} C(\vP_s, a) \,ds
 \right] \ge -\int_0^T \e^{-\rho s} c \,ds,
 \end{align*} and conversely for any $\xi \in \U(T)$,
 \begin{align*}
J^{\xi}(T,\vp,a) & \le \E^{\vp,a} \left[ \int_0^T \e^{-\rho s} C(\vP_s, \xi_s)\, ds \right]
   \le \int_0^T \e^{-\rho s} c \, ds.
 \end{align*}
Since \begin{align*} \int_0^T \e^{-\rho s} c \, ds \le \left\{ \begin{array}{cc}
 c T & \text{ when }\rho = 0; \\ c/\rho & \text{ when }\rho > 0, \end{array}\right.
 \end{align*}
we see that when $\rho > 0$ those bounds are even uniform in $T$.
\end{rem}

\begin{rem}\label{rem:opt-stopping-as-special-case}
One may extend the above analysis to cover the slightly more general case where $K(a,b,\vp)$
are allowed to be negative, as long as we assume that for \emph{any} chain $a_0, a_1, \ldots,
a_n$, $a_i \in \A$ we have
$$
K(a_0, a_1, \vp) + K(a_1, a_2, \vp) + \ldots + K(a_n, a_0, \vp) > k_0 > 0,$$ uniformly. This
condition implies that repeated switching is unprofitable and guarantees that the number of
switches along any path is finite with probability one. Then taking $\A' = \{ 0, \Delta_1,
\ldots, \Delta_A \}$ and for any $i\in\A$, $K(0,\Delta_i, \vp) = -\sum_{j\in E} H(i, j) \pi_j$,
$K(\Delta_i, 0, \vp) = +\infty$, $C(\vp,\Delta_i) = 0$, one may imbed the optimal stopping
problems studied in \cite{BS06} and \cite{LS07} in our framework. Namely, it is easy to see
that in this case
\begin{align}
U(T, \vp, 0) = \sup_{\tau \in \S(T), \xi_1 \in \A} \E^{\vp, a}\left[ \int_0^\tau \e^{-\rho s}
C(\vP_s, 0)\, ds + \e^{-\rho \tau} H(\xi_1, M_\tau) \right].
\end{align}
In that sense, our model is a direct extension of optimal stopping problems for hidden Markov
models with Poissonian observations.
\end{rem}

\begin{rem}\label{rem:qvi}
Using the dynamic programming principle developed in Proposition~\ref{prop:dp} one expects that the value function $U$ is the unique weak solution of  a coupled system of QVIs (quasi-variational inequalities)
\begin{equation}
\begin{split}
-\frac{\partial}{\partial T}U (T,\vp,a) + \mathcal{A} U  (T,\vp,a) -\rho U  (T,\vp,a) + C(\vp,a) &\leq 0 \\
U  (T,\vp,a) &\geq \mathcal{M} U  (T,\vp,a)\\
\left(-\frac{\partial}{\partial T}U  (T,\vp,a) + \mathcal{A} U  (T,\vp,a) -\rho U  (T,\vp,a) + C (\vp,a)\right)(U ( (T,\vp,a))-\mathcal{M}U  (T,\vp,a))&=0.
\end{split}
\end{equation}
Here $\mathcal{A}$ is the infinitesimal generator of the process $\vP$ given by (2.9).
$\mathcal{A}$ is a first order integro-differential operator. Note that the differential operators do not differentiate with respect to $a$, therefore for each $a$ we obtain a different QVI. These QVIs are coupled by the action of the intervention operator $\mathcal{M}$.

One could attempt to
numerically solve the above system of QVIs. However, the theoretical basis for the QVI formulation
requires justification, in particular in terms of the regularity of the value function $U$.
Typically one must pass to the realm of viscosity solutions to make progress; in contrast in the next section we will develop another a more direct characterization of the value function (see Proposition~\ref{eq:hat-wn}). In Section~\ref{sec:opt-strategy} we will use this characterization to develop the regularity properties of $U$, which helps us describe an optimal control. The more direct characterization of the value function in Proposition~\ref{eq:hat-wn} also provides us a numerical method for numerically solving for the value function.
\end{rem}

%According to Proposition \ref{prop:dp}, to solve \eqref{def:V} one needs to solve the system of
%coupled optimal stopping problems that constitute \eqref{eq:dyn-prog}. The second part of the
%propositions shows that the solution of these optimal stopping problems can be used to
%explicitly construct an optimal strategy for V.

\subsection{First Jump Operator}
The following Proposition \ref{eq:hat-wn} shows that the value function $U$ satisfies a second
dynamic programming principle, namely $U$ is the fixed point of the first jump operator
$\hat{L}$. This representation will be used in our numerical computations in Section
\ref{sec:examples}. Let us introduce a functional operator $L$ whose action on test functions
$V$ and $H$ is given by
\begin{multline}\label{eq:def-L}
L (V,H) (T,\vp,a)= \sup_{t \in [0,T]} \E^{\vp,a}\bigg[\int_0^{t \wedge \sigma_1}\! \e^{-\rho
s}C(\vP_s,a)\, ds \\ +1_{\{t <\sigma_1\}}\e^{-\rho t}H(T-t,\vP_t,a) +\e^{-\rho \sigma_1}1_{\{t
\geq \sigma_1\}} V(T-\sigma_1,\vP_{\sigma_1},a)\bigg].
\end{multline}
Observe that $L$ is clearly monotone in both of its function arguments. Moreover, we have
\begin{equation}\label{eq:L-optimal-stopping}
\begin{split}
L (V,H) (T,\vp,a)&= \sup_{\tau \in \mathcal{S}(T)} \E^{\vp,a}\bigg[\int_0^{\tau \wedge
\sigma_1}\!\e^{-\rho s}C(\vP_s,a) \,ds+1_{\{\tau<\sigma_1\}}\e^{-\rho t}H(T-\tau,\vP_{\tau},a)
\\&+\e^{-\rho \sigma_1}1_{\{\tau \geq \sigma_1\}} V(T-\sigma_1,\vP_{\sigma_1},a)\bigg],
\end{split}
\end{equation}
which follows as a result of the characterization of the stopping times of piecewise
deterministic Markov processes (Theorem T.33 \cite{bremaud}, and Theorem A2.3 \cite{davis93})
which state that for any $\tau \in \S(T)$, $\tau \wedge \sigma_1 = t\wedge \sigma_1$ for some
constant $t$.

Let us introduce another monotone functional operator by $$\hat{L}V \triangleq L(V,\M V).$$
\begin{prop}\label{eq:hat-wn} $U$ is the smallest fixed point of $\hat{L}$ that is larger than $U_0$. Moreover, the
following sequence which is constructed by iterating $\hat{L}$,
\begin{equation}\label{def:W}
W_0 \triangleq U_0, \quad W_{n+1} \triangleq \hat{L} W_n, \quad n \in \mathbb{N},
\end{equation}
satisfies $W_n \nearrow U$ (pointwise).
\end{prop}

\begin{proof}
\textbf{Step 1.} Recall that $\hat{L}$ is a monotone operator and that
\begin{align*}
W_1(T,\vp,a) = L(U_0, \M U_0)(T,\vp, a) & \ge \E^{\vp,a}\bigg[\int_0^{T \wedge
\sigma_1}\!\e^{-\rho s}C(\vP_s,a)
\,ds+\e^{-\rho \sigma_1}1_{\{T \geq \sigma_1\}} U_0(T-\sigma_1,\vP_{\sigma_1},a)\bigg] \\
 &= U_0(T, \vp, a) = W_0(T,\vp,a).
 \end{align*}
Therefore $(W_n)_{n \in \mathbb{N}}$ is an increasing sequence of functions. Denote the
pointwise limit of this sequence by $W=\sup_n W_n$. This limit is a fixed point of $\hat{L}$:
\begin{equation}\label{eq:shwfxdppt}
\begin{split}
W(T,\vp,a)&=\sup_{n \in \N} W_{n}(T,\vp,a)
\\&=\sup_{n \in \N}\sup_{t \in [0,T]}\E^{\vp,a}\bigg[\int_0^{t \wedge \sigma_1}\e^{-\rho s}C(\vP_s,a)ds+1_{\{t <\sigma_1\}}\e^{-\rho t}\M W_{n-1}(T-t,\vP_t,a)
\\&\qquad +\e^{-\rho \sigma_1}1_{\{t \geq \sigma_1\}} W_{n-1}(T-\sigma_1,\vP_{\sigma_1},a)\bigg]
\\&=\sup_{t \in [0,T]}\sup_{n \in \N}\E^{\vp,a}\bigg[\int_0^{t \wedge \sigma_1}\e^{-\rho s}C(\vP_s,a)ds+1_{\{t <\sigma_1\}}\e^{-\rho t}\M W_{n-1}(T-t,\vP_t,a)
\\&\qquad +\e^{-\rho \sigma_1}1_{\{t \geq \sigma_1\}} W_{n-1}(T-\sigma_1,\vP_{\sigma_1},a)\bigg]
=\hat{L} W (T,\vp,a),
\end{split}
\end{equation}
where the last line follows from the monotone convergence theorem. In fact it is the smallest
of the fixed points of $\hat{L}$ that is greater than $U_0=W_0$, which is a result of the
following induction argument: suppose that $\tilde{W} \ge U_0$ is another such fixed point.
Then $\tilde{W}=\hat{L} \tilde{W} \geq \hat{L} U_0=W_1$. On the other hand, if $\tilde{W} \geq
W_n$, then $\tilde{W}=\hat{L} \tilde{W} \geq \hat{L} W_n=W_{n+1}$. Now taking the supremum of
both sides we have that $\tilde{W} \geq W$.

\textbf{Step 2.} We will now show that $W$ is a fixed point of $\G$, hence $W \geq U$ as a
result of Proposition~\ref{prop:dp}. First, we will show that $W \geq \G W$. Let us construct
an increasing sequence of functions by $u_0=\M W$, $u_{n+1}=L(u_n,\M W)$, $n \in \N$. It can be
shown that $u_n$  can be written as
\begin{equation}\label{eq:u-n}
u_{n}(T,\vp,a)=\sup_{\tau \in \mathcal{S}(T)} \E^{\vp,a} \left[ \int_0^{\tau \wedge \sigma_n}
\!\!\e^{-\rho s} C(\vP_s, a) \, ds + \e^{-\rho\, \tau \wedge \sigma_n } \M W(T-(\tau \wedge
\sigma_n), \vP_{\tau \wedge \sigma_n}, a) \right],
\end{equation}
see e.g.\ Proposition 5.5 in \cite{bdk05}. Taking $n \to\infty$ we find that the monotone limit
$u = \lim_{n\uparrow \infty} u_n$ satisfies $u = \G W$. Now, we can show that $W \geq \G W$
using induction. From step 1, we know that $W= L(W, \M W)$, therefore $W \geq \M W=u_0$ (since
stopping immediately may not be optimal in \eqref{eq:L-optimal-stopping}). On the other hand,
if $W \geq u_n$, then since $L(\cdot, \M W)$ is a monotone operator, we have that $W=L(W, \M
W)\geq L(u_n, \M W)=u_{n+1}$. This implies that $W \geq u_n$ for all $n \in \N$. Therefore, $W
\geq \G W =\sup_{n} u_n$.

Let us show the reverse inequality: $W \leq \G W$. As a result of the monotone convergence
theorem we have that $\G W = \sup_{n \in \N} \G W_n$. Clearly $\G W_n \geq \hat{L} W_n$ since
$W_n \geq \M W_{n-1}$, and the set of stopping times that we are taking a sup over is smaller
than $\mathcal{S}(T)$. Therefore, $\G W_n \geq W_{n+1}$. Since we can repeat this argument for
all $n$, $$\G W =\sup_{n \in \N}\G W_n \geq \sup_{n \in \N} W_{n+1}=W.$$

\textbf{Step 3.} We will now show that $W \leq U$ (which together with the result of step 2,
shows that $W=U$). On the one hand, using the strong Markov property of $(\vP,\xi)$, the value
function $U$ can be shown to be a fixed point of $\hat{L}$ (see Proposition 5.6 in
\cite{bdk05}): recall that $U = \G U$ (the right-hand-side of which is an optimal stopping
problem) and compare with \eqref{eq:L-optimal-stopping}. On the other hand, from step 1 we know
that $W$ is the smallest fixed point of $\hat{L}$ greater than $U_0$. But this implies that $U
\geq W$.
\end{proof}

\begin{rem}
As a result of Fubini's theorem and using \eqref{eq:jumps-of-vP} and \eqref{def:m} we can write
$\hat{L}$ as
\begin{multline}
\label{eq:J-expectations} \hat{L} V (T,\vp,a)=  \sup_{0\leq t \leq T} \bigg\{\bigg(\sum_{i \in
E} m_i(t,\vp)  \bigg)
 \cdot \e^{- \rho t } \M
V\left(T-t,\vx( t, \vp ), a \right) \\
+  \int_{0}^{t} \e^{- \rho u}     \sum_{i \in E} m_i(u,\vp) \cdot \Bigl(  C(\vx(u, \vp), a) +
\lambda_i \cdot S_i V(T-u, \vx(u, \vp), a) \Bigr) du\bigg\},
\end{multline}
in terms of the operator
\begin{align}
\label{def:S} S_i w(t, \vp, a) \triangleq \int_{\R^d} w \left(t, \left(\, \frac{  \lambda_1
f_1(y) \pi_1 }{ \sum_{j \in E} \lambda_j f_j(y) \pi_j }, \ldots, \frac{  \lambda_m f_m(y) \pi_m
}{ \sum_{j \in E} \lambda_j f_j(y) \pi_j } \right), a\right) f_i(y) \nu(dy), \quad \text{for $i
\in E$.}
\end{align}
This implies that one can numerically compute $\hat{L}V$ by performing the \emph{deterministic}
optimization on the right-hand-side of \eqref{eq:J-expectations}.
\end{rem}

\section{Regularity of the Value Function and an Optimal Strategy}\label{sec:opt-strategy}
In this section we will analyze the regularity of the value function $U$, which will lead to
the construction of an optimal strategy. This is done by analysis of two auxiliary sequences of
functions converging to $U$. We first begin by studying $U_0$.

\begin{lemm}\label{lemm:u0cont}
The function $U_0$ defined in (\ref{def:U-0}) is convex in $\vp$.
\end{lemm}

\begin{proof}
Let us define a functional operator $I$ through its action on a test function $w$ by $I w=
L(w,0)$, that is,
\begin{equation}\label{eq: defn-I}
\begin{split}
Iw(T,\vp,a)&= \E^{\vp,a}\left[\int_0^{\sigma_1\wedge T}\e^{-\rho t}C(\vP_t,a)dt +1_{\{\sigma_1
\leq T\}}\e^{-\rho \sigma_1} w(T-\sigma_1,\vP_{\sigma_1},a)\right]
\\&= \int_0^{T}\e^{-\rho u} \sum_{i \in E} m_{i}(u,\vp) \cdot \left[C(\vec{x}(u,\vp),a)+\lambda_i \cdot S_i w(T-u, \vec{x}(u,\vp), a)\right]du.
\end{split}
\end{equation}
As a result of the strong Markov property of $\vP$ we observe that $U_0$ is a fixed point of $I$, and if we define
\begin{equation}\label{def:k-n}
k_{n+1}(T,\vp,a)= I k_n(T,\vp,a), \quad k_0(T,\vp,a)=0, \quad T \in \R_+, \vp \in D, a \in \A
\end{equation}
then $k_n \nearrow U_0$, see Proposition 1 in \cite{CostaDavis89}.  We will divide the rest of
the proof into two parts. In the first part we will show that $k_n$ converges to $U_0$
uniformly. In the second part we will argue that for all $n \in \N$, $k_n$ is convex. Suppose
both of the above claims have been proved and let $\eps >0$. Then for any  $\vp_1, \vp_2 \in D$
\begin{equation}
\begin{split}
U_0(T,\alpha \vp_1+ (1-\alpha) \vp_1,a)& = U_0(T,\alpha \vp_1+ (1-\alpha) \vp_1,a) -k_n(T,
\alpha \vp_1+(1-\alpha) \vp_2,a) \\ & \qquad +k_n(T, \alpha \vp_1+(1-\alpha) \vp_2, a)\\
&\leq \eps+ \alpha k_n(T,\vp_1,a)+(1-\alpha) k_n(T,\vp_2,a)
\\&\leq 2 \eps + \alpha U_0(T,\vp_1,a)+(1-\alpha)U_0(T,\vp_2,a),
\end{split}
\end{equation}
in which the last two inequalities follow since for $n>N(\eps)$ large enough,
$|U_0(T,\vp,a)-k_n(T,\vp,a)|< \eps$ for all $\vp \in D$. Since $\eps$ was arbitrary the
convexity of $\vp \rightarrow U_0(T,\vp,a)$ follows.

\textbf{Step 1}. Using strong Markov property we can write $k_n$ as (cf. \eqref{eq:u-n})
\begin{equation}
k_n(T,\vp,a)= \E^{\vp}\left[\int_0^{\sigma_n\wedge T}\e^{-\rho t}C(\vP_t,a)\,dt\right].
\end{equation}
As a result,
\begin{equation}\label{eq:kn-U0}
\begin{split}
|U_0(T,\vp,a)&-k_n(T,\vp,a)| \leq  \E^{\vp}\left[1_{\{T>\sigma_n\}}\int_{\sigma_n}^{T}\e^{-\rho
t}|C(\vP_t,a)|dt\right]
\\& \leq \E^{\vp}\left[1_{\{T>\sigma_n\}}\e^{-\rho \sigma_n} c \int_0^{T-\sigma_n}\e^{-\rho t}dt \right]
\\ &\leq c\, T\, \P^{\vp}\{T>\sigma_n\} \\ &\leq cT \E^{\vp}\left[1_{\{T>\sigma_n\}}(T/\sigma_n)\right] \leq  c\,
T^2\cdot  \E^{\vp}\left[1/\sigma_n\right].
\end{split}
\end{equation}

The conditional probability of the first jump satisfies $\P^{\vp}\{\sigma_1>t|M\}=\e^{-I(t)}$.
Therefore,
\begin{equation}\label{eq:fst-est}
\begin{split}
\E^{\vp}\left[\e^{-u \sigma_1}|M\right]=\E^{\vp}\left[\int_{\sigma_1}^{\infty}u\e^{-u t}\ ,dt
\Big| \, M\right] &=\int_0^{\infty}\P^{\vp}\{\sigma_1 \leq t | M \} \cdot u\e^{-ut}\,dt
\\& = \int_0^{\infty}\left[1-\e^{-I(t)}\right] u\e^{-ut}\, dt \\
&\leq \int_0^{\infty}\left[1-\e^{-\bar{\lambda} t}\right]u\e^{-u t}dt =
\frac{\bar{\lambda}}{\bar{\lambda}+u},
\end{split}
\end{equation}
where $\bar{\lambda}=\max_{i \in E}\lambda_i$, see \eqref{def:I-t}. Since the observed process
$X$ has independent increments given $M$, it readily follows that $\E^{\vp}\left[\e^{-u
\sigma_n}|M\right] \leq \bar{\lambda}^n/(\bar{\lambda}+u)^n$, which immediately implies that
\begin{equation*}%\label{eq:sec-set}
\E^{\vp}\left[\e^{-u \sigma_n}\right] \leq
\left(\frac{\bar{\lambda}}{\bar{\lambda}+u}\right)^n.
\end{equation*}
Also, since $1/\sigma_n=\int_0^{\infty}\e^{-\sigma_n u} du$, an application of Fubini's theorem
together with the last inequality yield
\begin{equation}
\E^{\vp} \left[\frac{1}{\sigma_n}\right] \leq \int_0^{\infty}\left(\frac{\bar{\lambda}}{\bar{\lambda}+u}\right)^n du =\frac{\bar{\lambda}}{n-1}, \quad n \geq 2.
\end{equation}
The uniform convergence of $k_n$ to $U_0$ now follows from (\ref{eq:kn-U0}) and
\eqref{eq:fst-est}.

\textbf{Step 2}. Here, we will show that $(k_n)_{n \geq 0}$ is a sequence of convex functions.
This result would follow from an induction argument once we show that the operator $I$ maps a
convex function to a convex function.

Let us assume that $\vp \rightarrow w(T,\vp,a)$ is a convex function for all $T \geq 0$.
Therefore, we can write this convex mapping as $\vp \rightarrow w(T-u,\vp,a)= \sup_{k \in K_u}
\alpha_{k,0}(T-u)+\alpha_{k,1}(T-u)\pi_1+\cdots+\alpha_{k,m}(T-u) \pi_m$, for some constants
$\alpha_{k,j} \in \mathbb{R}$ and countable sets $K_u$. Then using $x_i(t,\vp) =
m_i(t,\vp)/\sum_{j\in E} m_j(t,\vp)$ and the second equality in \eqref{eq: defn-I} we obtain
\begin{equation}
\begin{split}
I w(T,\vp,a)&=\int_0^{T}\e^{-\rho u} \sum_{i \in E} c_{i} m_i(u,\vp) \, du+
\int_{0}^{T}\e^{-\rho u} \sum_{i \in E} \lambda_i m_i(u,\vp) \cdot
\\& \hspace{-0.25in} \cdot\left[\int_{\R^d} \!\sup_{k \in K_u} \left(\alpha_{k,0}(T-u)+\sum_{j \in E} \alpha_{k,j}(T-u)\frac{\lambda_j f_j(y) m_{j}(u,\vp)}{\sum_{l \in E} \lambda_l f_l(y) m_{l}(u,\vp)} \right)f_i(y) \nu(dy)\right]du
\\& =\int_0^{T}\e^{-\rho u} \sum_{i \in E} c_{i} m_i(u,\vp) \, du
\\&+ \int_0^{T}\e^{-\rho u}\left[\int_{\R^d}\sup_{k \in K_u} \left(\sum_{j \in E}[\alpha_{k,j}(T-u)+\alpha_{k,0}(T-u)] \lambda_j
f_j(y)m_j(u,\vp)\right)\nu(dy)\right]du.
\end{split}
\end{equation}
Since $\vp \rightarrow m(u,\vp)$ is linear in $\vp$ (see \eqref{def:m}) and the supremum of
linear functions is convex, the convexity of $\vp \rightarrow I w(T,\vp,a)$ follows.

\end{proof}

\begin{lemm}\label{lemm:cont-U0}
$U_0(T,\vp,a)$ is continuous as a function of its first two variables.
\end{lemm}

\begin{proof}
The proof will be carried out in two parts. In the first part we will show that  $\vp \rightarrow U_0(T,\vp,a)$, is Lipschitz on $D$. In the second part we will show that $T \rightarrow U_0(T,\vp,a)$ is Lipschitz uniformly in $\vp$. But these two imply that $(T,\vp) \rightarrow U_0(T,\vp,a)$ is continuous for all $a \in \A$ since
\begin{equation}\label{eq:cont}
\begin{split}
|U_0(T,\vp,a)-U_0(S,\vec{p},a)|&=|U_0(T,\vp,a)-U_0(T,\vec{p},a)+U_0(T,\vec{p},a)-U_0(S,\vec{p},a)|
\\& \leq R(T,a) |\vp-\vec{p}|+ \tilde{R}(a) |T-S|, \;\quad \vp, \vec{p} \in D;\, T, S \in \R_+,
\end{split}
\end{equation}
in which $R(T,a)$ and $\tilde{R}(a)$ are the Lipschitz constants above.

\textbf{Step 1}. The idea is to use the convexity of $U_0$. Unfortunately, the convexity of
$\vp \rightarrow U_0(T,\vp,a)$ implies that this function is Lipschitz only in the
\emph{interior} of $D$. In what follows we will show that $\vp \rightarrow U_0(T,\vp,a)$ is the
restriction of a convex function $\vp \rightarrow \tU(\vp)$ whose domain is strictly larger
than $D$, which implies the Lipschitz continuity of $\vp \rightarrow U_0(T,\vp,a)$ also on the
boundary of the region $D$. To this end let us define the functional operator $\tilde{I}$
through its action on a test function $w$ as
\begin{equation*}
\tilde{I} w (T, \vec{p}, a)=\int_0^{T}\e^{-\rho u} \sum_{i \in E} m_{i}(u,\vp) \cdot
\left[C(\vec{x}(u,\vec{p}),a)+\lambda_i \cdot S_i w(T-u, \vec{x}(u,\vec{p}), a)\right]du,
\end{equation*}
for $\vec{p} \in \tilde{D}, T \in \R_+, a \in \A$ in which
\begin{equation*}
\tilde{D}= \left\{\vec{p} \in \mathbb{R}^m_+ \colon \sum_{i \in E} p_i \leq 2 \right\}.
\end{equation*}
Note that $\tilde{I}$ is nothing but an extension of the operator $I$ we defined in the proof
of Lemma~\ref{lemm:u0cont}. Let us define
\begin{equation*}
\tilde{k}_{n+1}(T,\vec{p},a)= \tilde{I} \tilde{k}_n(T,\vec{p},a), \quad
\tilde{k}_0(T,\vec{p},a)=0, \qquad T \in \R_+, \vec{p} \in \tilde{D}, a \in \A.
\end{equation*}
Using the very same arguments as in the proof of Lemma~\ref{lemm:u0cont}, we can show that
$\vec{p}\rightarrow \tilde{k}_n(T,\vec{p},a)$ is convex for all $n$, and this sequence of
functions uniformly converges to a convex limit $\vec{p} \rightarrow \tU(T,\vec{p},a)$.
Clearly, $\tilde{k}_n(T,\vec{p}, a)=k_n(T,\vec{p},a)$ when $\vec{p} \in D$. As a result
$U_0(T,\vec{p},a)=\tU(T,\vec{p},a)$ on $D$. Since $\tU(T,\vec{p},a)$ is locally Lipschitz in
the interior of $\tilde{D}$ (as a result of its convexity), we see that $\vec{p} \rightarrow
U_0(T,\vec{p},a)$ is Lipschitz on the compact domain $D$.

\textbf{Step 2}. The Lipschitz property of $U_0$ with respect to time (uniformly in $\vp$)
follows from
\begin{equation}
|U_0(T,\vp,a)-U_0(S,\vp,a)| \leq \E^{\vp,a} \left[ \int_S^{T}\e^{-\rho t} |C(\vP_t,a)| \, dt
\right] \leq c |T-S|.
\end{equation}
\end{proof}

\begin{lemm}\label{lem:Wconv}
For all $a \in \A$, $T \in \R_+$, $(W_n(T,\vp,\cdot))_{n \in \N}$, defined in $\eqref{def:W}$,
form a sequence of convex functions. Moreover, for  each $a \in \A$ and $n \in \N$, the
function $(T,\vp) \rightarrow W_n(T,\vp,a)$ is continuous.
\end{lemm}

\begin{proof}
The proof of the convexity of $\vp \rightarrow W_{n}(T,\vp,a)$ is similar to the proof of
convexity of $\vp \rightarrow k_n(T,\vp,a)$, which is defined in the proof of
Lemma~\ref{lemm:u0cont}, see Part II of that proof.

The continuity proof on the other hand parallels the continuity proof  for $(T,\vp) \rightarrow
U_0(T,\vp,a)$ which we carried out above. The proof of the uniform Lipschitz continuity of
$W_n$ with respect to time is similar to the corresponding proof for $U$ in Lemma
\ref{lem:U-lips-T} below.
\end{proof}

\begin{rem}
The value function $U$ is convex in $\vp$, since as a function of $\vp$, $U$ is the upper
envelope of convex functions $(W_n)$.
\end{rem}

\begin{lemm}\label{lem:U-lips}
The value function $U$ is Lipschitz continuous in $\vp$,
\begin{equation}
|U(T,\vp_1,a)-U(T,\vp_2,a)| \leq R(T,a) |\vp_1-\vp_2|, \qquad \vp_1, \vp_2 \in D;\, T \leq T_0;
\, a \in \A,
\end{equation}
where the positive constant $R$ depends on $T$ and $a$.
\end{lemm}

\begin{proof}
The proof parallels Step 1 of the proof of Lemma~\ref{lemm:cont-U0}. Again a convex sequence of
functions is constructed, converging upwards to an extension of $U$ on $\tilde{D}$ (each
element in this sequence is an extension of $W_n$ onto the larger domain.). Here, the
convergence is not uniform but monotone. The result still follows since the upper envelope of
convex functions is convex, so that the limit is convex and therefore Lipschitz in $\vp$ on the
original domain $D$.
\end{proof}

\begin{lemm}\label{lem:U-lips-T}
The value function $U$ is continuous in $T$ uniformly in the other variables, namely
\begin{equation}
|U(T,\vp,a)-U(S,\vp,a)| \leq c |T-S|, \quad\text{for any}\quad \vp \in D; \, T , S \in (0,T_0];
\, a \in \A.
\end{equation}
\end{lemm}
\begin{proof}
Fix $S>T$. Let $\xi^S, \xi^T$ be $\eps$-optimal strategies for $U(S,\vp,a)$ and $U(T,\vp,a)$
respectively. Then, taking $\tilde{\xi}^S = \xi^T 1_{[0, T]} + \xi^T_T 1_{(T, S]}$ we have
\begin{align*}
 U(S,\vp,a)-U(T,\vp,a) & \geq J^{\tilde{\xi}^S}(S,\vp,a) - (J^{\xi^T}(T,\vp, a) + \eps) \\
& = \E^{\vp,a}\left[ \int_T^S \e^{-\rho s } C(\vP_s,\xi^T_{T}) \, ds \right] -\eps \geq
-\e^{-\rho T}(S-T) c - \eps.
\end{align*}
On the other hand, using the strong Markov property of $(\vP,\xi^S)$,
\begin{align*}
 U(S,\vp,a)-U(T,\vp,a) & \leq  J^{\xi^S}(S,\vp,a) +\eps - J^{\xi^S \cdot 1_{[0,T]}}(T,\vp,a) \\
& = \E^{\vp,a} \left[ \int_T^{S} \e^{-\rho s} C(\vP_s, \xi^S_{s}) \, ds - \sum_{k \colon \tau_k
> T} \e^{-\rho \tau_k} K(\xi^S_{k-1},\xi^S_k, \vP_{\tau^S_k})
\right] +\eps \\
& \leq \E^{\vp,a} \left[ \e^{-\rho T} \!\int_T^S \e^{-\rho (s-T)} C(\vP_s, \xi^S_{s}) \, ds \right] + \eps \\
& \leq \e^{-\rho T} (S-T) c + \eps,
\end{align*}
Since $\eps$ was arbitrary, we therefore conclude that $|U(T,\vp,a)-U(S,\vp,a)| \leq c |T-S|$
as desired.
\end{proof}

\begin{lemm}
For each $a \in\A$ and $n$, the function $(T,\vp) \rightarrow V_n(T,\vp,a)$ is continuous.
\end{lemm}

\begin{proof}
We proved in Lemma \ref{lemm:cont-U0} that $(T,\vp) \rightarrow U_0(T,\vp,a)$ is continuous.
Furthermore, observe that the operator $\M$ preserves continuity: if for all $a \in \A$,
$(T,\vp) \rightarrow V(T,\vp,a)$ is continuous then for $(T_1,\vp_1)$ and $(T_2,\vp_2)$ close
enough
\begin{equation}
|\M V(T_1,\vp_1,a)-\M(T_2,\vp_2,a)| \leq \max_{ b \in \A, b \neq a}|V(T_1,\vp_1,b)-V(T_2,\vp_2,b)|
\end{equation}
is small.

The rest of the proof follows due to the properties of the operator $\G$ in \eqref{def:G}.
Indeed, $\G w(\cdot, \cdot, a)$ defines an optimal stopping problem for $\vP$ with terminal
reward function $\M w(\cdot, \cdot, a)$. As shown in Corollary 3.1 of \cite{LS07} (see also
Remark 3.4 in \cite{BS06}), when $\M w$ is continuous, then the value function $\G w$ of this
optimal stopping problem is also continuous. Therefore, by induction, $V_{n+1} = \G V_n$ is
continuous.

\end{proof}

\begin{cor}\label{cor:cont}
The value function $U(\cdot,\cdot,a)$ is continuous for all $a \in \A$. Moreover,
$(V_n(\cdot,\cdot,a))_{n \in \mathbb{N}}$ defined in (\ref{defn:Un}) and
$(W_n(\cdot,\cdot,a))_{n \geq 0}$, defined in Proposition~\ref{eq:hat-wn}, both converge to
$U(\cdot,\cdot,a)$ uniformly for all $a \in \A$.
\end{cor}

\begin{proof}
Lemmas~\ref{lem:U-lips} and \ref{lem:U-lips-T} imply the continuity of $U(\cdot,\cdot,a)$ (also see \eqref{eq:cont}). Now
the rest of the statement of the corollary follows from Dini's theorem, which states that
pointwise convergence of continuous functions to a continuous limit implies uniform convergence
on compacts.
\end{proof}

Using Corollary~\ref{cor:cont} we obtain the following explicit existence result about an
optimal strategy for $U$:
\begin{prop}
Let us extend the value functions $U_0$ and $U$ so that
\begin{equation}
U_0(T,\vp,a)=U(T,\vp,a)=0, \quad T \in [-\eps,0),\, \vp \in D, \; a \in \A,
\end{equation}
for some strictly positive constant $\eps$.
Let us recursively define a strategy $\xi^* = (\xi_0, \tau_0; \xi_1, \tau_1, \ldots)$ via
$\xi_0 = a, \tau_0 = 0$ and
\begin{align}
\left\{
\begin{aligned}
\tau_{k+1} &= \inf \{s \in [\tau_k,T] \colon U(T-s,\vP(s),\xi_k) = \M U(T-s,\vP(s),\xi_k) \}; \\
\xi_{k+1} & = d_{\M U}(T-\tau_{k+1},\vP({\tau_{k+1}}),\xi_k), \qquad k = 0,1,\ldots,
\end{aligned} \right.
\end{align}
with the convention that $\inf \emptyset =T+\eps$.
Then $\xi^*$ is an optimal strategy for \eqref{def:V}, i.e.,
\begin{equation}\label{eq:dec-p-str}
U(T,\vp,a)=\E^{\vp,a}\left[\int_0^{T}\!\e^{-\rho s}C(\vP(s),\xi^*_s) \,ds -\sum_{k: \tau_k \leq
T}\e^{-\rho \tau_k} K(\xi_k,\xi_{k+1},\vP({\tau_k}))\right].
\end{equation}
\end{prop}

\begin{proof}
We will show that for $n=1,2,\ldots$
\begin{multline}\label{eq:ind-opt-st}
\E^{\vp,a}\left[\int_0^{\tau_n}\!\e^{-\rho s}C(\vP(s),\xi_s)ds -\!\sum_{k=0}^{n-1}\e^{-\rho
\tau_k} K(\xi_k,\xi_{k+1},\vP({\tau_k}))\right] \\ = U(T,\vp,a)-\E^{\vp,a}\left[\e^{-\rho
\tau_n} U (T-\tau_n,\vP({\tau_n}),\xi_n)\right].
\end{multline}
Suppose that \eqref{eq:ind-opt-st} is true. Then
\begin{equation}
\begin{split}
\E^{\vp,a}&\left[\int_0^{T}\!\e^{-\rho s}C(\vP(s),\xi_s) \,ds -\sum_{k=0}^{n-1}\e^{-\rho
\tau_k} K(\xi_k,\xi_{k+1},\vP({\tau_k}))\right]
\\ &= U(T,\vp,a)-\E^{\vp,a}\left[\e^{-\rho \tau_n} U
(T-\tau_n,\vP({\tau_n}),\xi_n)\right]+\E^{\vp,a}\left[\e^{-\rho \tau_n} U_0
(T-\tau_n,\vP({\tau_n}),\xi_n)\right].
\end{split}
\end{equation}
Taking the limit as $n \rightarrow \infty$ and using bounded convergence theorem and $\tau_n
\rightarrow T+\eps$, we have that
\begin{align*}
 U(T,\vp,a) &=\E^{\vp,a}\left[\int_0^{T}\e^{-\rho
s}C(\vP(s),\xi_s)ds -\sum_{k}\e^{-\rho \tau_k} K(\xi_k,\xi_{k+1},\vP({\tau_k}))\right] \\
& \leq \E^{\vp,a}\left[\int_0^{T}\e^{-\rho s}C(\vP(s),\xi_s)ds -\sum_{k: \tau_k \leq
T}\e^{-\rho \tau_k} K(\xi_k,\xi_{k+1},\vP({\tau_k}))\right],
\end{align*}
since $K(a,b,\vp) > 0$, and equation (\ref{eq:dec-p-str}) follows.

To establish \eqref{eq:ind-opt-st} we proceed by induction. The functions $U(\cdot,\cdot,a)$
and $\M U(\cdot,\cdot,a)$ are continuous by Corollary~\ref{cor:cont}. As a result the stopping
time
\begin{equation}
\tau_1= \inf \left\{s \in [0,T]: U(T-s,\vP(s),a)=\M U(T-s, \vP(s), a) \right\},
\end{equation}
satisfies
\begin{equation}
\E^{\vp,a}\left[\int_0^{\tau_1}\e^{-\rho s}C(\vP(s),a)ds+\e^{-\rho \tau_1} \M U
(T-\tau_1,\vP({\tau_1}),a)\right]=U(T,\vp,a),
\end{equation}
see e.g.\ Proposition 5.12 in \cite{bdk05}. Rearranging and using $\xi_1 = d_{\M
U}(T-\tau_{1},\vP({\tau_{1}}),a)$,
\begin{equation}\label{eq:aar}
\E^{\vp,a}\left[\int_0^{\tau_1}\!\e^{-\rho s}C(\vP(s),\xi_0)\,ds -\e^{-\rho \tau_1}
K(\xi_0,\xi_1,\vP({\tau_1}))\right]= U(T,\vp,a)-\E^{\vp,a}\left[\e^{-\rho \tau_1} U
(T-\tau_1,\vP({\tau_1}),\xi_1)\right],
\end{equation}
proving \eqref{eq:ind-opt-st} for $n=1$. Perhaps we should emphasize the dependence on $T$ on
the left-hand-side of \eqref{eq:aar} by inserting $T$ as another superscript above $\E$ (we are
conditioning on the strong Markov process $t \rightarrow (T-t,\vP_t,\xi_t)$). Although we are
not going to implement this for notational consistency/convenience, one should keep this point
in mind when reading the rest of the proof.

Assume now that for some $n \ge 1$ (\ref{eq:ind-opt-st}) is satisfied; we will prove that it
also holds when we replace $n$ by $n+1$. Since $\tau_n$'s are all hitting times we have that
$\tau_{n+1}=\tau_n+\tau_1 \circ \theta_{\tau_n}$.
\begin{equation}\label{eq:ind-opt-st-2}
\begin{split}
&\E^{\vp,a}\left[\int_0^{\tau_{n+1}}\e^{-\rho s}C(\vP(s),\xi_s)ds -\sum_{k=0}^{n}\e^{-\rho
\tau_k} K(\xi_k,\xi_{k+1},\vP({\tau_k}))\right]=\E^{\vp,a} \bigg[\int_0^{\tau_n}\e^{-\rho
s}C(\vP(s),a)ds
\\&-\sum_{k=0}^{n-1}\e^{-\rho \tau_k} K(\xi_k,\xi_{k+1},\vP({\tau_k}))+\e^{-\rho \tau_n}\E^{\vP(\tau_n),\xi_n}\bigg[\int_0^{\tau_1}\e^{-\rho s}C(\vP(s),\xi_0)ds -\e^{-\rho \tau_1} K(\xi_0,\xi_1,\vP({\tau_1}))\bigg]\bigg]
\end{split}
\end{equation}
Using \eqref{eq:aar} we can then write
\begin{equation}\label{eq:ind-opt-st-3}
\begin{split}
\E^{\vp,a}\bigg[\e^{-\rho \tau_n}&\E^{\vP(\tau_n),\xi_n}\big[\int_0^{\tau_1}\e^{-\rho
s}C(\vP(s),\xi_0)ds -\e^{-\rho \tau_1} K(\xi_0,\xi_1,\vP({\tau_1}))\big]\bigg]
\\&=\E^{\vp,a}\left[\e^{-\rho \tau_n}U(T-\tau_n,\vP(\tau_n),\xi_n)-\e^{-\rho \tau_{n+1}}U(T-\tau_{n+1},\vP(\tau_{n+1}),\xi_{n+1})\right].
\end{split}
\end{equation}
Using (\ref{eq:ind-opt-st-2}) and \eqref{eq:ind-opt-st-3} together with the induction
hypothesis, we obtain \eqref{eq:ind-opt-st} when $n$ is replaced by $n+1$.

\end{proof}

Let
\begin{align}
\label{def:stop-cont-regions}
\begin{aligned}
\mathcal{C}_s(a) &\triangleq \left\{ \vp \in D : U(s, \vp,a) > \M U(s,\vp,a) \right\}, \\
\Gamma_{s}(a) &\triangleq \left\{ \vp \in D : U(s, \vp, a) = \M U(s, \vp,a ) \right\}
\end{aligned}
\end{align}
denote the continuation and switching regions for initial policy $a$ with $s$ time units until
maturity. The switching region can further be decomposed as the union $\cup_{b \in \mathcal{A}
} \Gamma_{s}(a,b)$ of the regions defined as
\begin{align}
\label{def:accept-reject-regions}
\begin{aligned}
\Gamma_{s}(a,b) &\triangleq \left\{ \vp \in D : U(s, \vp,a) = U(s,\vp,b) - K(a,b,\vp) \right\},
\qquad b \in \mathcal{A},
\end{aligned}
\end{align}

The results in the previous section imply that to solve \eqref{eq:dyn-prog} with initial
horizon of $T$, one maintains the initial policy $a$ and observes the process $\vP$ until time
$\tau_1=\tau_1(T)$, whence it enters the region $\Gamma_{T-\tau_1}(a)$. At this time, if
$\vP_{\tau_1}$ is in the set $\Gamma_{T-\tau_1}(a,b)$ we take $\xi_1 = b$; that is, we select
the $b$'th policy in the policy set $\mathcal{A}$. The boundaries of $\Gamma_{s}(a,b)$ are
termed switching boundaries and provide an efficient way of summarizing the optimal strategy of
the controller. We plot these curves in our examples in Section \ref{sec:examples}.

\section{Extensions}\label{sec:extend}
\subsection{Infinite Horizon Formulation}\label{sec:infinite-horizon}
In many practical settings, the controller does not have a natural horizon for her strategies.
In such cases it is more appropriate to consider infinite-horizon setting. Due to
time-homogeneity, the infinite-horizon problem is stationary in time, reducing the dimension by
one. In particular, the optimal strategy can be simplified with a single switching-boundary
plot, as $\Gamma_s(a)$'s are independent of $s$.

For $\rho > 0$, let
\begin{align}\label{eq:inf-horizon-defn}
V_\rho(\vp, a) = \sup_{\xi \in \U(\infty)} \E^{\vp,a} \left[ \int_0^\infty \e^{-\rho t}
C(\vP(t), \xi_t) \, dt - \sum_k \e^{-\rho \tau_k} K(\xi_{k-1}, \xi_k, \vP({\tau_k}))
\right].\end{align} Here $\U(\infty)$ denotes the admissible strategies that satisfy $\E^{\vp,
a} \left[ \sum_k \e^{-\rho \tau_k} K(\xi_{k-1}, \xi_k, \vP({\tau_k})) \right] < \infty$.

The next proposition shows that the infinite horizon problem can be uniformly approximated by
the finite horizon problems. In fact, the convergence is exponentially fast in the time horizon
$T$.
\begin{prop}
There exists a constant $R$ such that
\begin{equation}
|U(T,\vp,a) - V_\rho(\vp,a)| \leq \e^{-\rho T} R.
\end{equation}
\end{prop}

\begin{proof}
Let $\xi^T$ be an $\eps$-optimal strategy of $U(T,\vp,a)$ and $\tilde{\xi}^T = \xi^T(t)
1_{[0,T]} + \xi^T_T 1_{(T, \infty)} \in \U(\infty)$. Then
\begin{align*}
V_\rho(\vp,a) - U(T,\vp,a) & \geq \E^{\vp,a} \left[ \int_T^\infty \e^{-\rho s} C(\vP_s,
\tilde{\xi}^T_s) \, ds \right] - \eps \\
& \geq -\e^{-\rho T}\int_0^\infty \e^{-\rho s} c \, ds - \eps \geq -\e^{-\rho T} c/\rho - \eps. \\
\intertext{On the other hand, using an $\eps$-optimal control $\xi^\infty$ of $V_\rho(\vp,a)$,}
V_\rho(\vp,a) - U(T,\vp,a)& \leq \E^{\vp,a} \left[ \int_T^\infty \e^{-\rho s} C(\vP_s,
\xi^\infty_s) \, ds - \sum_{k \colon \tau_k > T} \e^{-\rho \tau_k} K(\xi^\infty_{k-1},
\xi^\infty_k,
\vP_{\tau_k^\infty}) \right] + \eps \\
& \leq \E^{\vp,a} \left[ \e^{-\rho T} \E^{\vP_T,\xi^\infty_T} \left[ \int_0^\infty \e^{-\rho s}C(\vP_s,\xi^\infty_T)\, ds  \right] \right] + \eps \\
& \leq \e^{-\rho T} \tilde{R} + \eps,
\end{align*}
for some constant $\tilde{R}$ where the last line used the fact that the inner term, which is
the infinite-horizon counterpart of $U_0$, is uniformly bounded on the compact domain $D \times
\A$. Taking $R = \max( \tilde{R}, c/\rho)$ the proposition follows.
\end{proof}

The
characterization of the value function of the infinite horizon problem, which we give below, follows along same lines as in Section \ref{sec:opt-strategy}.
\begin{prop}
$V_\rho$ is the smallest fixed point of the operator $\hat{L}_\rho(V) \triangleq L_\rho(V,\tM
V)$ where
$$ L_\rho(V,H)(\vp, a) = \sup_{t \ge 0} \E^{\vp,a} \left[ \int_0^{t \wedge \sigma_1} \e^{-\rho s}
C(\vP_s, a) \, ds + 1_{\{t < \sigma_1\}} \e^{-\rho t} H(\vP_t, a) + 1_{\{t \ge \sigma_1\}}
\e^{-\rho \sigma_1} V(\vP_{\sigma_1}, a) \right]$$ and
$$ \tM V(\vp,a) = \max_{b \in \A, b \neq a} \left\{V(\vp, b) - K(a,b,\vp) \right\}.$$
\end{prop}

Note that $\hat{L}_\rho$ is given by
\begin{multline}
\label{def:J-prime} \hat{L}_\rho  w(\vp,a) =  \sup_{t \ge 0}  \biggl\{\left(\sum_{i \in E}
m_i(t,\vp)\right)
 \cdot \e^{- \rho t }  \cdot  \tM w\left(\vx( t, \vp ), a \right)
 \\ + \int_{0}^{t} \e^{- \rho u} \sum_{i \in E}  m_i(u,\vp) \left[ C(\vx(u, \vp), a) + \lambda_i \tilde{S}_i w( \vx(u, \vp), a) \right] \, du \biggr\} ,
\end{multline}
where
\begin{align*}
%\label{def:S}
\tilde{S}_i w( \vp, a) \triangleq  \int_{\R^d} w \left( \left(\frac{  \lambda_1 f_1(y) \pi_1 }{
\sum_{j \in E} \lambda_j f_j(y) \pi_j }, \ldots, \frac{  \lambda_m f_m(y) \pi_m }{ \sum_{j \in
E} \lambda_j f_j(y) \pi_j } \right), a\right) f_i(y) \nu(dy), \quad i =1, \ldots, m,
\end{align*}
for a bounded function $w(\cdot,\cdot)$ defined on $ D \times \A $ only. The optimal stopping
time for $V_\rho$ is now the first entrance time $\tau_0( \vp )$ of the process $\vP$ to the
time-stationary region
%\begin{align}
%\label{def:stopping-region-for-infinite-horizon}
$$ \Gamma(a) = \left\{ \vp \in D : \, V_\rho(\vp, a) = \tM V_\rho(\vp, a)\right\}.$$
%\end{align}

To compute $V_\rho$ we define again
$$
W_0(\vp, a) = \E^{\vp,a} \Bigl[ \int_0^\infty \e^{-\rho s} C(\vP_s, a) \, ds \Bigr],\quad
\text{ and }\quad W_{n+1} = \hat{L}_\rho W_n. $$ Then as in Section \ref{sec:opt-strategy}, it
can be shown that $W_n \nearrow V_\rho$, and $W_n$ can be computed numerically by using
\eqref{def:J-prime}.

%Also, let us define the average-cost functional
%\begin{align}\label{eq:ave-cost-defn}
%V_0(a) = \lim_{T\to\infty} \sup_{\xi \in \U(T)}  \frac{1}{T}\E^{\vp,a} \left[ \int_0^T
%\e^{-\rho t} C(\vP(t), \xi_t) \, dt - \sum_k \e^{-\rho \tau_k} K(\xi_{k-1}, \xi_k,
%\vP({\tau_k})) \right].
%\end{align}
%Because there is no discounting and the strong Markov process $\vP$ is mixing, $V_0$ is
%independent of the initial condition $\vp$. The following result characterizes $V_0$ as a limit
%of $V_\rho$.
%\begin{prop}
%$\lim_{\rho \to 0} \rho V_\rho(\vp, a) = V_0(a)$. Moreover, suppose that there exists a pair
%$(w, \beta)$, $w \in \mathcal{C}(D \times \A)$, $\beta \in \R^\A$ satisfying the coupled
%optimal stopping problem
%\begin{align*} w(\vp, a) = \sup_{t \ge 0} \E^{\vp,a} \left[ \int_0^{t
%\wedge \sigma_1} \{ C(\vP(s), a) - \beta(a) \}\, ds + 1_{\{t < \sigma_1\}} \tM w(\vP(t), a) +
%1_{\{t \ge \sigma_1\}} w(\vP(\sigma_1), a) \right].\end{align*}

% Then $\beta(a) = V_0(a)$ and an optimal strategy for \eqref{eq:ave-cost-defn} is given by $
% \tau_1 =\inf \{ t \ge 0 \colon w(\vP(t), a) \le \tM w(\vP(t), a) \}$, $\tau_n = \tau_1 \circ \theta_{\tau_{n-1}}$, $\xi_n = \argmax \tM(
% w(\vP(\tau_n),a)$,.
%\end{prop}

%This proposition follows immediately by using the results of Section \ref{sec:opt-strategy} and
%applying Theorems 1 and 2, as well as Lemma 4 in \cite{CostaAverage91} to our setting.

\subsection{Costs Incurred at Arrival Times}\label{sec:discrete-costs}
In many practical settings the arrivals of $X$ are themselves costly which leads us to consider
a running cost structure of the form
$$
\sum_{j=1}^{N(t)} \e^{-\rho \sigma_j} c_i(Y_j, a) 1_{\{ M_{\sigma_j} = i\}}, $$ where $c_i:
\R^d \times \A \mapsto \R$ (with $\int_{\R^d} c^+_i(y,a) \nu_i (dy) < \infty $ for all $i \in
E, a \in \A$) is the cost incurred upon an arrival of size $Y_j$ when the controller has policy
$a$ in place and the environment is $M_{\sigma_j}=i$. Above $N(t)$ is the number of arrivals by
time $t$, and $(\sigma_j, Y_j)$ are the arrival times and marks respectively. As an example,
see Section \ref{sec:customer-call} below.
%
%As an example, in a customer service call center, $X$ represents the customer calls whose
%frequency $\lambda$ and urgency $Y$ are modulated by the telephone traffic conditions $M$. The
%controller chooses the staffing level; larger number of staff reduces cost of servicing each
%call but has larger labor cost.

In the latter case, setting $C(y,\vp,a) = \sum_{i} \pi_i c_i(y,a)$ one deals with the objective
function
\begin{align}
\label{def:U-alternative} \tilde{U} (T, \vp, a) \triangleq
 \sup_{ \xi \in \U(T)}
\E^{\vp,a} \left[
%\int_0^{\tau} \e^{- \rho t }   \left( \sum_{i \in E} c_i 1_{ \{  M_{t} = i  \} } \right) dt
\sum_{j=1}^{N(T)} \e^{-\rho \sigma_j} C(Y_j,\vP(\sigma_j),\xi_{\sigma_j}) - \sum_k \e^{- \rho
\tau_k} K(\xi_{k-1}, \xi_{k}, \vP({\tau_k}))  \right] ,
\end{align}
by solving the equivalent coupled stopping problem
\begin{align*}
%\label{def:V-alternative}
\tilde{U} (T, \vp,a) \triangleq
 \sup_{ \tau \in \s(T) }
\E^{\vp,a} \left[
%\int_0^{\tau} \e^{- \rho t }   \left( \sum_{i \in E} c_i 1_{ \{  M_{t} = i  \} } \right) dt
\sum_{j=1}^{N(\tau)} \e^{-\rho \sigma_j} C(Y_j,\vP({\sigma_j}),a)+ \e^{- \rho \tau} \M \hat{U}
\left(T-\tau, \vP({\tau }), a \right) \right] ,
\end{align*}
as in Proposition~\ref{def:V}. One can easily verify that the function $\tilde{U}$ is the
smallest fixed point greater than $U_0$ of the operator $\tilde{L}$ whose action on a test
function $w$ is
\begin{multline*}
\tilde{L} w ( T, \vp,a) = \sup_{0 \le t \le T} \biggl\{ \Bigl(\sum_{i \in E}  m_i(t,\vp) \Bigr)
 \cdot \e^{- \rho t }  \cdot  \M w\left(T-t,\vx( t, \vp ), a \right)   \\
+  \int_{0}^{t} \e^{- \rho u}     \sum_{i \in E}  m_i(u,\vp) \cdot \lambda_i
\left(\int_{\R^d}\! C(y,\vx(u,\vp),a) \nu_i(dy) +  S_i w(T-u, \vx(u, \vp), a) \right) du
\biggr\}.
\end{multline*}

\section{Numerical Illustrations}
\label{sec:examples}

Below we provide numerical examples illustrating our model based on the applications outlined
in Section \ref{sec:apps}. The numerical implementation proceeds by discretizing the time
horizon $[0,T]$ and then directly finding the deterministic supremum over $t$'s in
\eqref{eq:J-expectations}. Similarly, the domain $D$ is also discretized and linear
interpolation is used for evaluating the jump operator $S$ of \eqref{def:S}.  Because the
algorithm proceeds forward in time with $t=0,\Delta t, \ldots, T$, for a given time-step $t=
m\Delta t$, the right-hand-side in \eqref{eq:J-expectations} is known and one may obtain
$U(m\Delta t, \vp, a)$ directly. %Thus, only two runs are required: first to compute $U_0$ and
%then to implement \eqref{eq:J-expectations}.

On infinite horizon since there is no time-variable the dynamic programming equation
\eqref{def:J-prime} is coupled. Accordingly, one must use the iterative sequence of $W_n$, as
detailed in Section \ref{sec:infinite-horizon}. Namely, one first computes $W_0 = U_0$ by
iterating \eqref{def:k-n}, and then applies $\hat{L}_\rho$ several times to find a suitably
good approximation $W_n$.

% on the case studies in Section \ref{sec:catalogue}.

\subsection{Optimal Tracking of `On-Off' System}\label{sec:security-detection}
Consider a physical system (for example a military radar) that can be in two states
$E=\{1,2\}$. Information about the system is obtained via a point process $X$ that summarizes
observations. The controller wishes to track the state of the system by announcing at each time
$0 \le t \le T$ whether the current state is $a=1$ or $a=2$, $\A = \{1,2\}$. The controller
faces a penalty if her announcement is incorrect; namely a running benefit is assessed at rate
$c_{1}(1) \, dt $ (respectively, $c_{2}(2)\, dt$) if the controller declares $\xi_t=1$ and
indeed $M_t = 1$ (resp. $\xi_t=2$ and $M_t = 2$). If the controller is incorrect then no
benefit is received. Moreover, the controller faces fixed costs $K({1,2})$ (resp. $K({2,1})$)
from switching her announcement from state 1 to state 2. $K(a,b)$'s represent the effort for
disseminating new information, alerting other systems, triggering event protocols, etc. A case
in point is the alert announcements by the Department of Homeland Security regarding terrorist
threat level which receive major coverage in the media and have significant nationwide
implications with high associated costs. Thus, both in the case of an upgrade and in the case
of a downgrade, specific protocols must be followed by appropriate government and corporate
departments. These effects imply that alert levels should be changed only when significant
changes occur in the controller beliefs.

To illustrate we take without loss of generality $c_1(1) = c_2(2) = 1, c_1(2) = c_2(1) =0$ and
first consider $K(1,2) = K(2,1) = 0.05$, $\rho = 0$, $T=1$. We assume that $X$ is a simple
Poisson process with corresponding intensities $\lambda(M) = [1, 4]$, so that arrivals are much
more likely in the `alarm' state 2. Finally, the generator of $M$ is
\begin{align*}
Q = \begin{pmatrix} -1 & 1 \\ 3 & -3 \end{pmatrix},\end{align*} so that on average an alarm
should be declared $\lim_{t\to\infty} \P\{ M_t = 2\} = 25\%$ of the time.

Figure \ref{fig:two-state} shows the results, in particular the switching regions
$\Gamma_s(a,b)$. We observe a highly non-trivial dependence of the switching boundaries on time
to maturity. First, very close to maturity, no switching takes place at all, as the fixed
switching costs $K$ dominate any possible gain to be made. For small $s$, the no-switching
region $\mathcal{C}_s(a)$ is very large, because the controller is reluctant to change her
announcement close to maturity. On the other hand, we observe that the switching region in
policy 1 narrows between medium $s \sim 0.2$ and large $s$. This happens again due to the
finite horizon. With $s=0.2$, when the controller believes that $M_t=2$ with high probability,
it is unlikely that $M_t$ will change again before maturity, so that the optimal strategy is to
pay the switching cost $K(1,2)$ and plan to maintain policy 2 until expiration. On the other
hand, for large $s \ge 0.5$, even when $\P\{M_t = 2\}=1-\pi_1$ is quite large, the controller
knows that soon enough $M$ is likely to return to state $1$ (since $q_{2,1}$ is large); rather
than do two switches and track $M$, the controller takes a shortcut and continues to maintain
policy 1 (with the knowledge that her error is likely to be shortlived). This
``shortcircuiting'' will disappear only when $\pi_1$ is extremely small. Note that this
phenomenon is one-sided: because $q_{1,2}$ is small, the upper boundary $\Gamma_{s}(2,1)$ is
monotonically decreasing over time.

\begin{figure}[ht]
\begin{tabular*}{\textwidth}{lr}
\begin{minipage}{3.2in}
\center{\includegraphics[height=2.4in,width=3.2in]{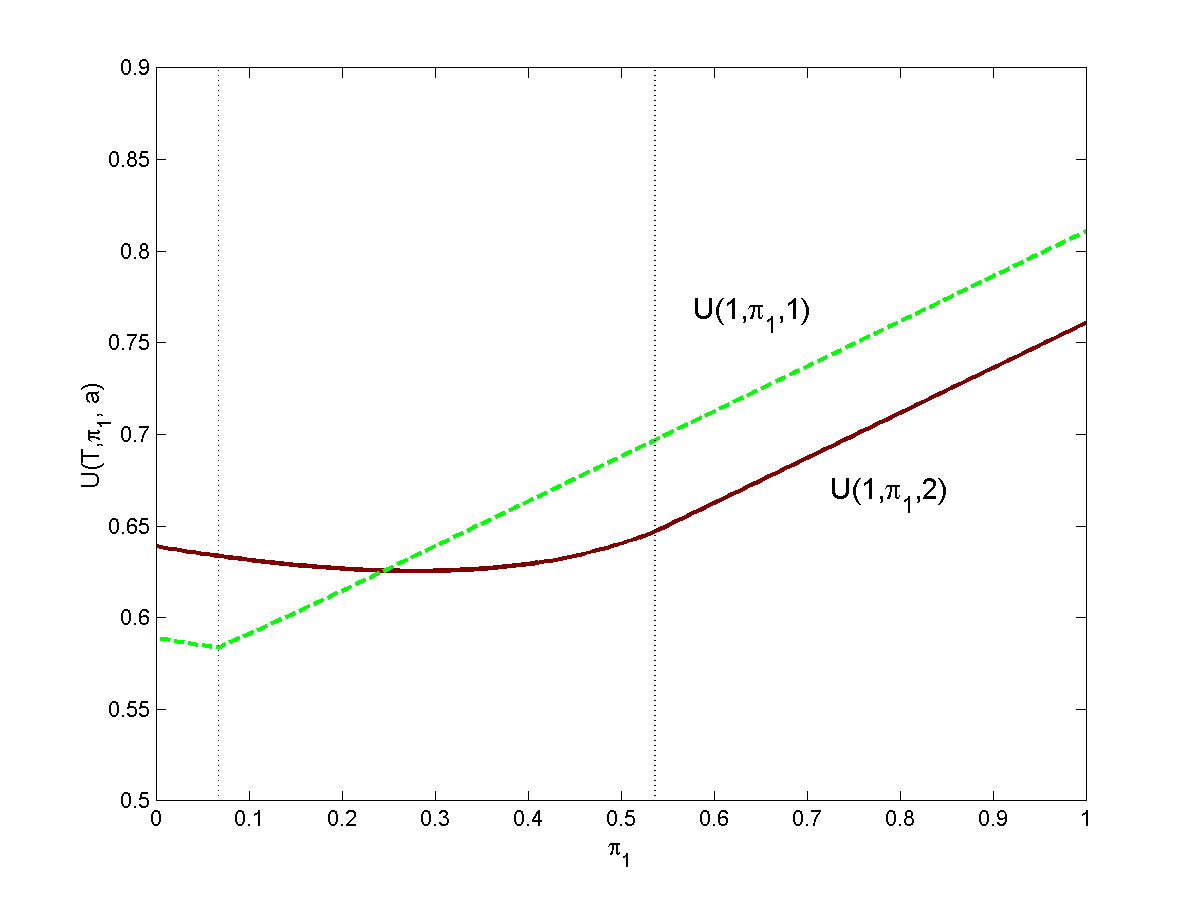}}

\end{minipage}
\begin{minipage}{3.2in}
\center{\includegraphics[height=2.4in,width=3.2in]{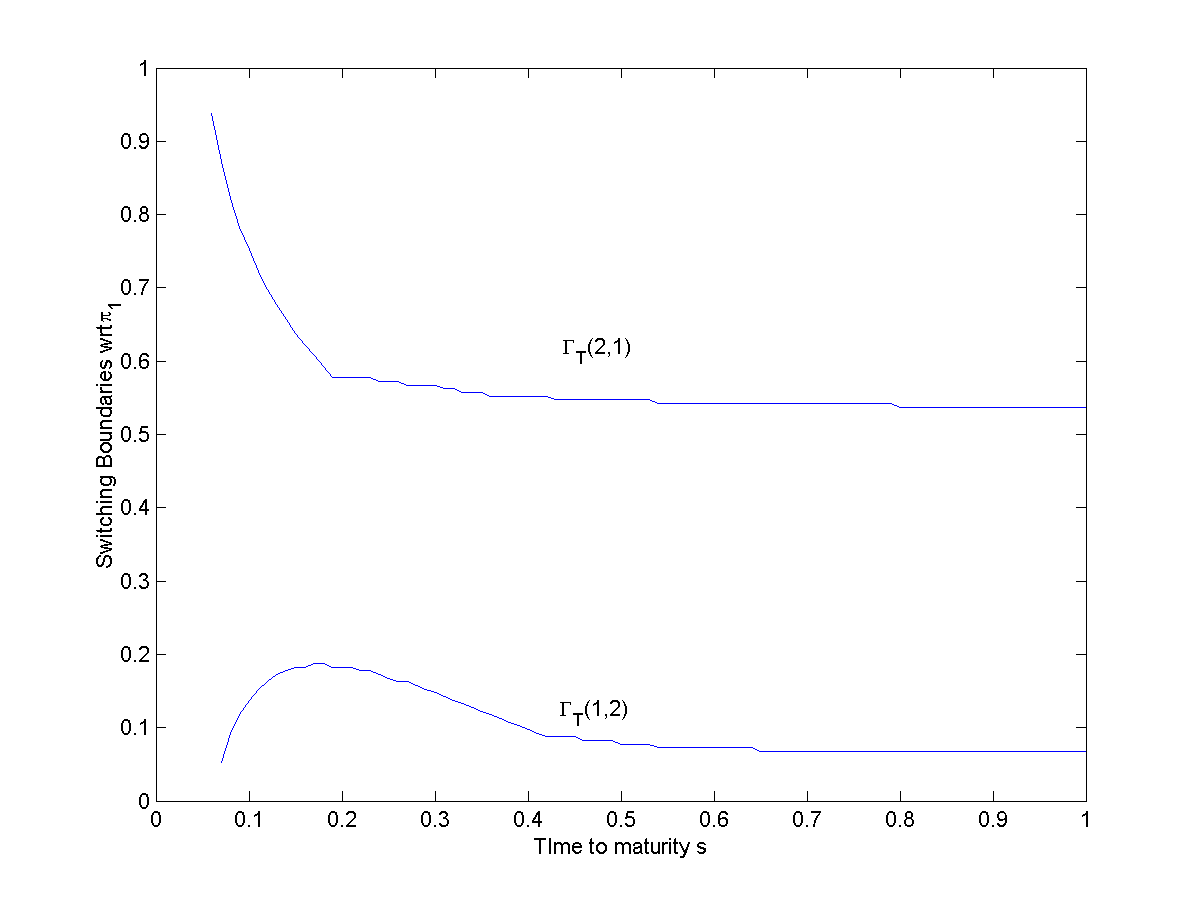}}

\end{minipage}
\end{tabular*}
\caption[Sequential Tracking]{Sequential tracking of a two-state Markov chain. The left panel
shows the value functions $U(T,\vp,\cdot)$, $a \in \{1, 2\}$, as a function of $\pi_1$ for
$T=1$. Recall that in this case $D = \{ (\pi_1, 1-\pi_1) \colon 0 \le \pi_1 \le 1 \}$. The
vertical lines indicate the boundary of $\Gamma_T(1,2)$ and $\Gamma_T(2,1)$. The right panel
shows the switching regions $\Gamma_{s}(a,b)$ (namely $\Gamma_{s}(1,2)$ is below the lower
curve and $\Gamma_{s}(2,1)$ is above the higher curve) as a function of time to maturity $s$.
\label{fig:two-state}}
\end{figure}

\subsection{Policy Making Example.}\label{sec:fed-target}
%To illustrate the above paradigm, we consider a problem in technology adoption.
The Federal Reserve Board (the Fed) has the task of adjusting the US monetary policy in
response to economic events. The Fed has authority over the overnight interest rates and
attempts to implement a loose monetary policy when the economy is weak, and a tight monetary
policy when the economy is overheating. Unfortunately, the current state of the economy $M$ is
never precisely known; thus the main task of the Fed is to estimate $M$ from various economic
information it collects. When the beliefs of the Fed change sufficiently, it will adjust its
monetary policy $\xi$. Such adjustments are expensive, since they are closely followed by
market participants and send out important signals to economic agents. Thus, beyond trying to
track $M$, the Fed also seeks \emph{stability} in its policies, in order not to disrupt
planning activities of businesses.

As can be seen from this description, this problem fits well into our tracking paradigm of
\eqref{def:U}. For concreteness, let $M = \{ M_t \}_{ t \ge 0}$ represent the current economy
with state space $E = \{1, 2, 3\} \equiv \{ Overheating, Growth, Recession \}$. The generator
of $M$ is taken to be $$ Q =
\begin{pmatrix} -4 & 3 & 1 \\  2 & -4 &  2 \\  0 &  3 & -3 \end{pmatrix}.$$ Thus, $M$
moves randomly between all three states (and we assumed that a recession cannot be immediately
followed by overheating). In the face of these three states, the Fed also has three policy
levels, namely its action set is $\A = \{0, 1, 2\} = \{ Tight, Normal, Accommodating \}$.

The cost function $C(\vp,a) = \sum_{i\in E} c_{i}(a) \pi_i$, is given by the matrix
$$ c_{i}(a) = c_{i,a} = \begin{bmatrix} 2 &  -1 &  -1 \\ 0 &  2 & 0 \\  -1 &  -1 & 0
\end{bmatrix}, \qquad \qquad i\in E, a \in \A.$$
The switching costs are given by $K(a,b) = 0.05 \cdot 1_{\{a\neq b\}}$ for $a,b \in \A$. The
observation process $X$ is a simple Poisson process with  $M$-modulated intensity
$\vec{\lambda} = [ \lambda_1, \lambda_2 , \lambda_3] = [1, 2, 5]$. Thus, the worse the economy
state, the more frequent are (negative) events observed by the Fed.

\begin{figure}%[ht]\hspace{-1.5in}
\begin{tabular*}{\textwidth}{lcr}
\begin{minipage}{2.1in}
%%\centering{\includegraphics[height=2.4in,width=2.1in,clip]{techAdopt1-0.png}}
\centering{\includegraphics[height=2.4in,width=2.1in,clip]{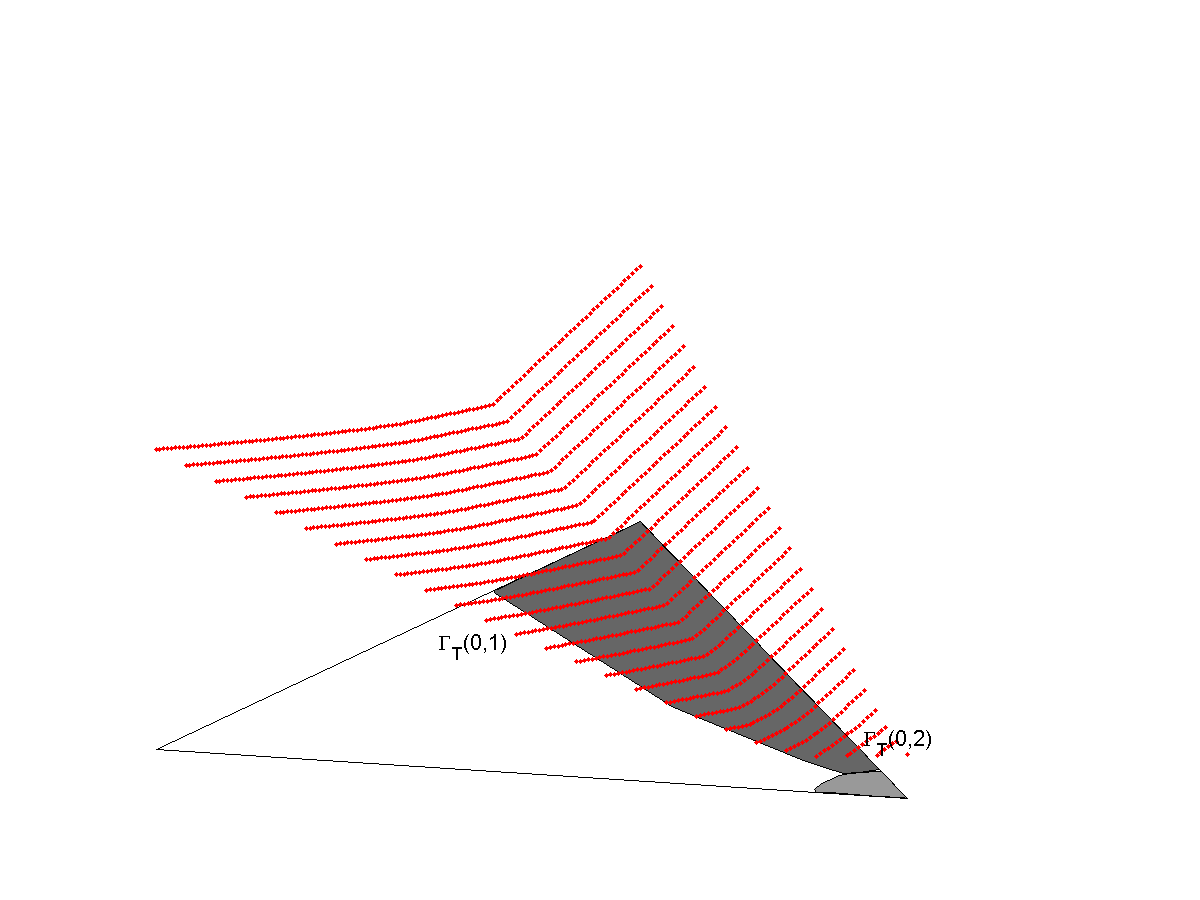} $a = 0$}

\end{minipage}
\begin{minipage}{2.1in}
%\centering{\includegraphics[height=2.4in,width=2.1in,clip]{techAdopt0-25.png}}
\centering{\includegraphics[height=2.4in,width=2.1in,clip]{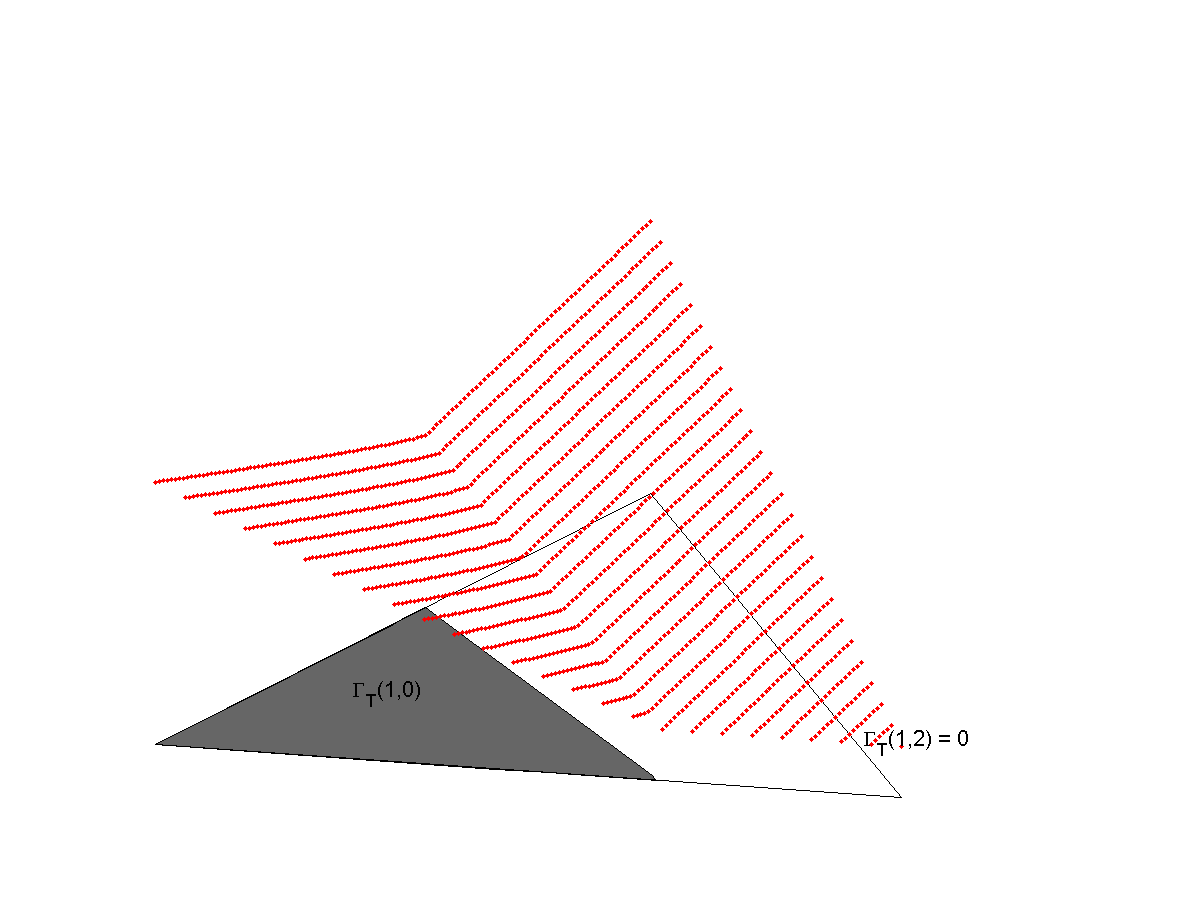} $a = 1$}

\end{minipage}
\begin{minipage}{2.1in}
%\centering{\includegraphics[height=2.4in,width=2.1in,clip]{techAdopt005.png}}
\centering{\includegraphics[height=2.4in,width=2.1in,clip]{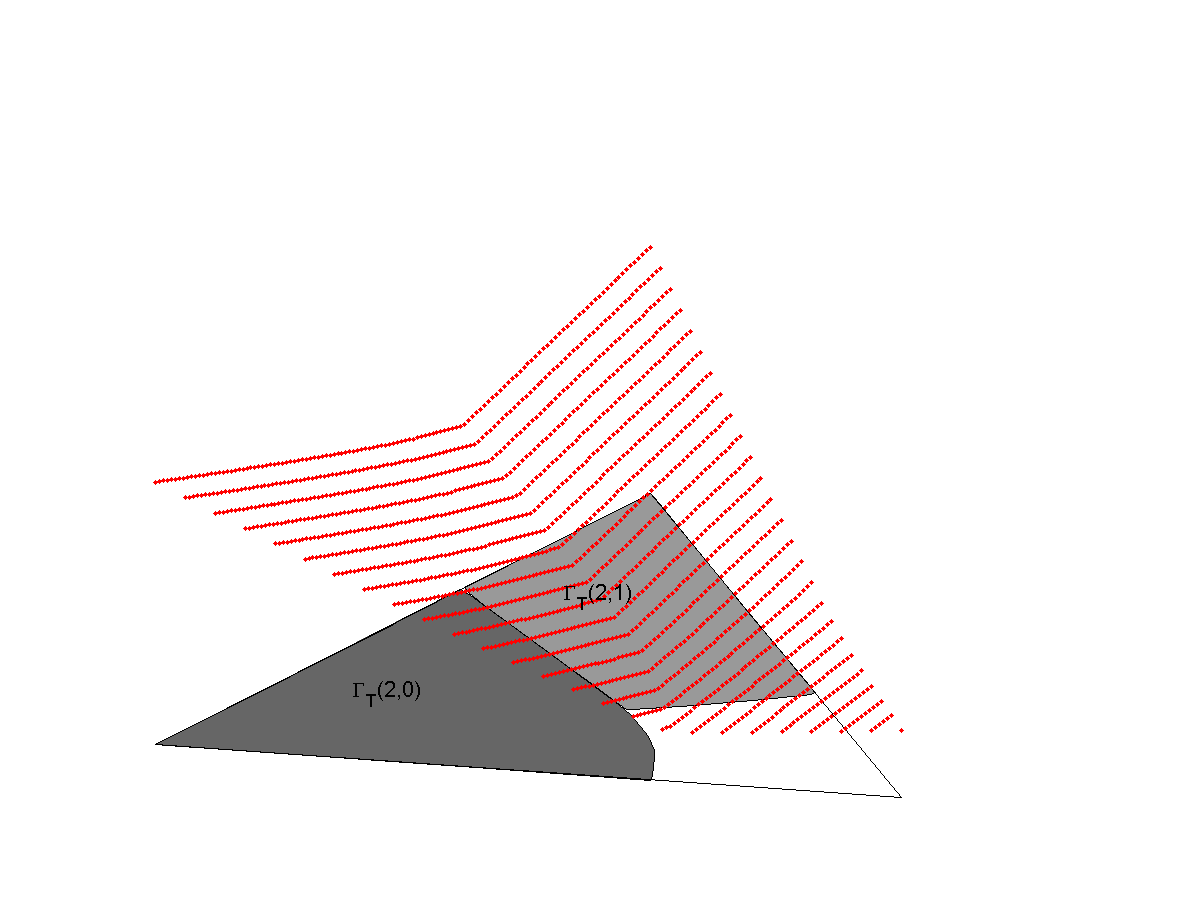} $a = 2$}

\end{minipage}
\end{tabular*}
\caption[Fed Example]{Value function $U(T,\vp,a)$ of the Fed policy-making example in Section
\ref{sec:fed-target} plotted together with the switching regions $\Gamma_T(a,\cdot)$ for each
current policy $a$.\label{fig:fed-target}}
\end{figure}

Figure \ref{fig:fed-target} illustrates the obtained results for $T=4$ and no discounting. The
triangular regions in Figure \ref{fig:fed-target} are the state space $D = \{ \vp \in \R_+^3
:\, \pi_{Ovr} + \pi_{Gro} + \pi_{Rec} =1\} $. The respective panels show how the initial
switching regions $\Gamma_{T}(a, \cdot)$ and value functions $U(T,\vp,a)$ depend on the current
policy $a$. Observe that because the penalty for not tracking recessions is small, starting out
in the `Normal' regime, the Fed will never immediately adopt an `Accommodating' policy,
$\Gamma_T(1,2)=\emptyset$. Similarly, because the penalty for missing an overheating economy is
very large, the switching regions into a `Tight' policy are large and conversely, the
continuation region $\mathcal{C}_T(0)$ is large. Also, observe that the value function appears
to be not differentiable at the boundaries. Finally, we stress that because of the final
horizon, this problem is again non-time-stationary and the solution (as well as $\Gamma_{T}(a,
\cdot)$) depends on remaining time $T$.

\subsection{Customer Call Center Example}\label{sec:customer-call}
Our last example illustrates the structure of the infinite horizon version together with a
different cost structure. We consider a call center application that employs a variable number
of servers to answer calls. The calling rate fluctuates and is modulated by the unknown
environment variable $M$. Having more servers decreases the per-call costs, but increases fixed
costs related to payroll overhead.

We assume that $M_t \in E = \{ Low, Med, High \}$ with a generator $$ Q =
\begin{pmatrix} - 1 & 1 & 0 \\ 1 & -2 & 1 \\ 0 & 1 & -1 \end{pmatrix}.$$
The observed process $X$ represents the actual received calls and is taken to be a compound
Poisson process with intensity $\lambda(M_t)$ and marks $Y_1, Y_2, \ldots$ that represent
intrinsic call costs. Suppose that $Y \in \{ 6, 12, 24\}$, and the distribution of $Y$ and
$\lambda$ is $M$-modulated:
$$ \nu_{i,j} = \PP\{ Y=y_j | M_t = i\} = \begin{pmatrix} 1/4 & 1/2 & 1/4 \\ 1/3 & 1/3 & 1/3 \\
1/4 & 1/4 & 1/2 \end{pmatrix}; \qquad \vec{\lambda} = [1 \quad 3 \quad 4].$$  Thus, as the
manager receives calls, she dynamically updates her beliefs about current state of $M$ based on
the intervals between call times and observed call types.

The call center manager can choose one of two strategies, namely she can employ either one or
two agents, $a \in \A = \{1, 2\}$. Employing $a$ agents leads to per-call costs of $c_1(Y, a) =
-Y/a$ and to continuously-assessed costs of $c_2(Y,a) =- (10 + 20a)$. Thus, when $\P\{M_t =
High\}$ is sufficiently high, it is optimal to employ both agents, otherwise one is sufficient.
Finally, switching costs for increasing or decreasing number of agents are set at $K(a, b) =
2$. Note that here all the costs are independent of $M$ (and hence of $\vP$).

We consider an infinite horizon formulation and take $\rho=0.5$. The parameter $\rho$ measures
the trade-off between minimizing immediate costs and having a long-term strategy that takes
into account future changes in $M$. Thus $\rho=0.5$ means that the horizon of the controller is
on the time-scale of two time periods. The overall objective is:
\begin{align*}
\sup_{\xi \in \U(\infty)} \E^{\vp,a} \left[ \sum_{j=1}^{\infty} \e^{-\rho \sigma_j}
c_1(Y_j,\xi_{\sigma_j}) + \int_0^\infty \e^{-\rho t}c_2(\xi_t) \,dt - \sum_k \e^{- \rho \tau_k}
 K(\xi_{k-1}, \xi_{k})  \right].
\end{align*}
Figure \ref{fig:call-center} shows the results, as well as a computed color-coded sample path
of $\vP$ which shows the implemented optimal strategy. The given path has four jumps and three
policy changes (two changes occur between jumps when $\vP$ enters $\Gamma(1,2)$, and one change
occurs at an arrival when $\vP$ jumps back into $\Gamma(2,1)$). Observe that in the absence of
new information, $\vP$ converges to the fixed point $\vp_\infty = [0.7, 0.23, 0.07]$ (the
invariant distribution of $\e^{Q - \Lambda}$), as can be seen from the flow of the paths in
Figure \ref{fig:call-center}.

% Do 3/2 ie just two policies or 2/3?
\begin{figure}[ht]
%\begin{tabular*}{\textwidth}{lr}
%\begin{minipage}{3.2in}
%\center{\includegraphics[height=2.4in,width=3.2in]{tmpbl-cc.png}}
%
%\end{minipage}
%\begin{minipage}{3.2in}
\center{\includegraphics[height=3in,width=4in]{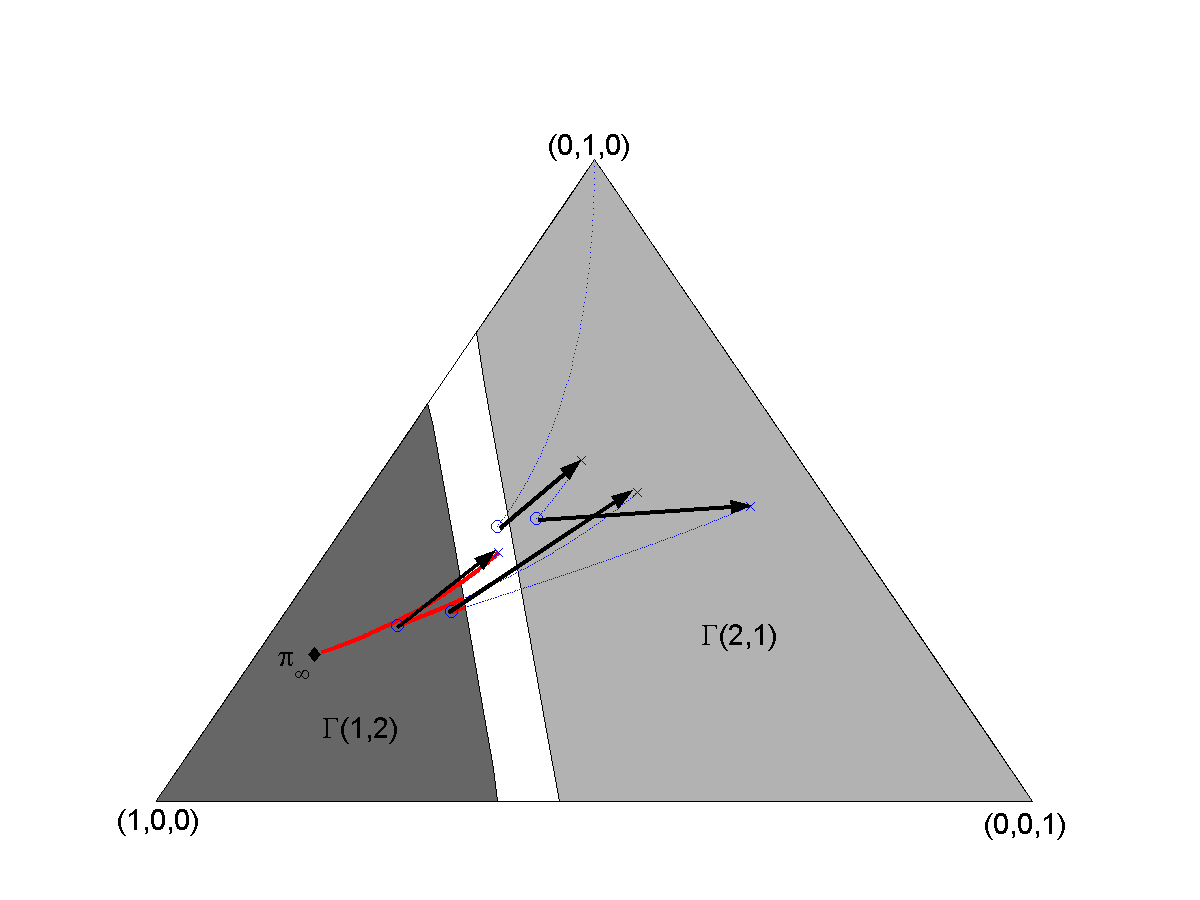}}
%
%\end{minipage}
%\end{tabular*}
\caption[Call Center]{Tracking the regime of a customer call center. We show a sample path of
$\vP$ inside the simplex $D = \{(\pi_1, \pi_2, \pi_3) : \pi_i \ge 0, \pi_1 + \pi_2 + \pi_3 = 1
\}$, as well as the corresponding optimal strategy. The initial state is $\vP_0 = (0,1,0)$ and
$\xi_0 = 1$. On this path we have $t\in[0,4]$ and the arrival pairs (corresponding to jumps of
$X$, recall $\vP$-dynamics in \eqref{eq:jumps-of-vP}) $(\sigma_\ell,Y_\ell)$ for $\ell=1,2,3,4$
are $(0.51, 2), (0.66, 3), (1.44, 1), (2.23, 2)$, respectively. The resulting optimal strategy
$\xi^*$ is color-coded: dashed line for $\xi^*_t = 1$, solid line for $\xi^*_t =
2$.\label{fig:call-center}}
\end{figure}

\appendix
\renewcommand{\thesection}{A}
\refstepcounter{section}
\makeatletter
\renewcommand{\theequation}{\thesection.\@arabic\c@equation}
\makeatother

\bibliography{tracking-references}
\bibliographystyle{abbrvnat}
\end{document}